\documentclass[11pt,reqno]{amsart}
\usepackage{amscd}
\usepackage{amsfonts}
\usepackage[all]{xy}
\usepackage{amsmath,amsthm,hyperref}
\usepackage{amsmath,amssymb,amsthm,latexsym}
\usepackage{color}
\usepackage{amsmath}
\usepackage{bm}
\usepackage{dsfont}
\usepackage{changes}
\usepackage{enumerate}
\usepackage{graphicx} 
\numberwithin{equation}{section}
\newtheorem{theorem}{Theorem}[section]
\newtheorem{proposition}[theorem]{Proposition}
\newtheorem{definition}[theorem]{Definition}
\newtheorem{corollary}[theorem]{Corollary}
\newtheorem{lemma}[theorem]{Lemma}
\usepackage{caption}
\usepackage{float}
\usepackage{graphicx}
\usepackage{CJK}
\usepackage{amsmath,amsfonts,mathrsfs,amssymb}
\usepackage{indentfirst}
\usepackage{caption}

\usepackage{multicol}
\usepackage{float}
\usepackage{longtable}
\usepackage[numbers,sort&compress]{natbib}

\usepackage{xcolor}
\usepackage{indentfirst}
\usepackage{amscd}
\usepackage{amsfonts}
\usepackage{amsmath,amsthm,hyperref}
\usepackage[all]{xy}
\usepackage{amsmath,amsthm,hyperref}
\usepackage{amsmath,amssymb,amsthm,latexsym}
\usepackage{amscd}
\theoremstyle{remark}
\newtheorem{remark}[theorem]{\bf Remark}

\newtheorem{example}[theorem]{\bf Example}

\newcommand{\R}{\mathbb{R}}
\newcommand{\rn}{\mathbb{R}^N}

\newcommand{\sn}{\mathbb{S}^N}
\newcommand{\C}{\mathbb{C}}
\newcommand{\Z}{\mathbb{Z}}

\newcommand{\B}{\mathbb{B}}
\newcommand{\BB}{\mathcal{B}}
\newcommand{\DD}{\mathcal{C}}
\newcommand{\T}{\mathbb{T}}
\newcommand{\M}{\mathbb{M}}
\newcommand{\CC}{\mathcal{C}}
\newcommand{\NN}{\mathcal{N}}
\newcommand{\<}{\langle}
\newcommand{\ra}{\rangle}
\newcommand{\dd}{\mathrm d} 
\newcommand{\TT}{\mathcal{T}}
\newcommand{\Tg}{\mathcal{T}_g}
\newcommand{\LL}{\widetilde{\Lambda}^*}

\newcommand{\HT}{\CJKfamily{hei}}

\newcommand{\W}{\mathcal{W}}

\newcommand{\TB}{\mathbb{T}^2_b}

\newcommand{\KB}{\mathbb{K}^2_{b}}
\newcommand*{\innerproduct}[2]{%
    \if@display
        \left\langle #1,#2\right\rangle
    \else
        \langle #1,#2 \rangle
    \fi}
\begin{document}

\title{Willmore Energy Estimates for Klein Bottles}
\author{Ruijie Ni, Peng Wang}
\address{School of Mathematics and Statistics, FJKLAMA, Key Laboratory of Analytical Mathematics and Applications (Ministry of Education), Fujian Normal University, Fuzhou, China}
\email{345784487@qq.com, niruijie721@sina.com}

\address{School of Mathematics and Statistics, FJKLAMA, Key Laboratory of Analytical Mathematics and Applications (Ministry of Education), Fujian Normal University, Fuzhou, China}\email{pengwang@fjnu.edu.cn, netwangpeng@163.com}

\date{\today}

\maketitle

\begin{abstract}
In this note, we provide an estimate for the Willmore energy of Klein bottles in $\rn$, building on the ideas of Li-Yau and Montiel-Ros for $2$-tori. In particular, we prove that the Willmore energy $\W(\phi)> 6\pi$ for any conformal branched immersion $\phi:\KB\rightarrow \sn$, where $\KB$ is a flat Klein bottle with $0.350\lesssim b \lesssim 0.755$. This confirms partially a conjecture of Kusner. Moreover, for a flat immersion $\psi:\KB\rightarrow \rn$, we derive a lower bound of $\W(\psi)$, which confirms a conjecture of Hirsch-M\"{a}der-Baumdicker for flat Klein bottles in $\rn$. We also construct a smooth family of flat Klein bottles in $\R^5$,  with Willmore energy between $6.912\pi$ and $8\pi$.
  \end{abstract}

  \vspace{2mm}

  \vspace{2mm}

 \section{introduction}
The Willmore problem for surfaces in space forms is an important topic in global differential geometry. Given a closed surface $M$ in Euclidean space $\rn$, its Willmore energy is defined by
$$
\mathcal{W}(M) = \int_M |H|^2\, \dd M,
$$
where $H$ denotes the mean curvature of the surface and $\dd M$ is the area form of $M$. Intuitively, the Willmore energy measures the total bending or curvature complexity of the surface, which can be tricked back to Germain.  The famous Willmore conjecture  \cite{W-1965} states that for any torus $T^2$ embedded in the three-dimensional sphere $\mathbb{S}^3$, its Willmore energy satisfies the inequality $\mathcal{W}(T^2) \geq 2\pi^2,$
with equality if and only if the torus is M\"{o}bius congruent to the Clifford torus. Li and Yau introduced the notion of   conformal area in \cite{L-Y} and  solved the Willmore conjecture when  $M$ is conformally equivalent to a flat torus with lattice generated by $\{(1,0),(a,b)\}$, where $0\leq a \leq \frac{1}{2}$ and $\sqrt{1-a^2}\leq b\leq 1$.
In \cite{M-R}, Montiel and Ros also employed the conformal area to estimate the Willmore energy and extended the scope of the result of Li and Yau to the set $A=\{(a,b)\mid 0\leq a\leq \frac{1}{2},b>0 , a^2+b^2\geq1\,\text{and}\,(a-\frac{1}{2})^2+(b-1)^2\leq \frac{1}{4}\}$. Similar results were obtained independently by Bryant (see  \cite{B-1988,B-2015}). In 2014, the Willmore conjecture was solved by Marques and Neves (\cite{M-N2}, \cite{M-N1},) using geometric measure theory and minimal surface methods, marking a significant advancement in geometric analysis. We refer to \cite{M-N1,WangWang2022} for recent surveys.

Closely related to the Willmore conjecture is the following conjecture proposed by Kusner: 
If a closed surface $M$ in $\sn$ has $\W(M)\le 6\pi$, then $M$ is either a sphere or a real projective plane; the latter happens if and only if $M$ is the Veronese minimal surface in $\mathbb S^4$ (up to M\"obius transformation of $\sn$)  \cite{K-1996}. This conjecture is mainly concerned with lower bounds for the Willmore energy of closed surfaces of arbitrary genus. In fact, Kusner obtained upper estimates for the infimum of the Willmore energy for orientable, closed surfaces in each genus, and more generally in each regular homotopy class of closed surfaces in $\mathbb{S}^3$, by constructing comparison surfaces in \cite{K-1989}. Indeed, the seminal works of Li-Yau \cite{L-Y} and Marques-Neves \cite{M-N2} give a positive answer to this conjecture for all closed surfaces in $\mathbb{S}^3$. However, lower estimates for the Willmore energy infimum of surfaces of genus $\geq1$ in higher co-dimension, remain open.

As a higher genus surface whose geometric properties have been extensively studied, the Klein bottle has seen many recent results on estimates of its Willmore energy. 
 Breuning, Hirsch and M\"ader-Baumdicker show that there exists a smooth conformal immersion $\tilde{f}_b: \mathbb{T}_b \to \mathbb{R}^4$ that is twistor holomorphic and is the oriented double cover of Klein bottle $\mathbb{K}_b$. And the corresponding immersed Klein bottle $f_b:\mathbb{K} _b \to \mathbb{R}^4$ is an embedded Willmore surface with Willmore energy $\W(f_b)=8\pi$ and Euler normal number $e(\nu)=\pm 4$. See \cite{B-H-M} for more details.
 In \cite{H-M}, Hirsch and M\"ader-Baumdicker proved that in the conformal class of $\widetilde{\tau}_{3,1}:K\to \mathbb S^4\subset \sn$,
the surface $\widetilde{\tau}_{3,1}$ is the unique minimizer of the Willmore energy in $\sn$, up to conformal transformations of $\sn$. They also conjecture that the unique minimizer among all immersed Klein bottles in $\sn,n\geq 4$, is $\tilde{\tau}_{3,1}$. In addition, for every $b>0$, let $\mathbb{T}_b=\mathbb{C}/\Gamma$ with $\Gamma=\langle 1, \mathrm{i}b\rangle$.

In this paper, we first use the conformal area method of Li-Yau \cite{L-Y} and Montiel-Ros \cite{M-R} to study the conformal classes of Klein bottles, represented by the flat Klein bottles $\KB$. More precisely, we reduce the problem to estimating the area of conformal immersions into spheres satisfying the balance condition, and obtain a lower bound for the conformal area of flat Klein bottles. The resulting estimate splits into a contribution from finitely many Fourier modes and a remainder term, denoted by $R$. By exploiting the symmetry induced by the deck transformation of the Klein bottle, together with the conformality condition, we further show that $R$ admits an explicit positive lower bound, which leads to the following theorem. 
  \begin{theorem} \label{thm conf imm}
    Let $\psi_K:\KB\longrightarrow \rn$ be a  branched conformal immersion. If  $b\in\left(b_1,b_2\right)$, then $\mathcal{W}(\psi_K)> 6\pi.$
   Here $b_1\approx0.350$ and $b_2\approx0.755$ are the two positive solutions to 
$
\frac{\pi^2b}{4b^2+1}\bigl(8+2R_*(b)\bigr)=6\pi,
$ where $R_*(b)$ is given in Theorem \ref{thm:improved-lower-bound}.
\end{theorem}

We then focus on flat Klein bottles into Euclidean spaces. Recall that Chen proved that every such immersion satisfies $\mathcal W>2\pi^2$ \cite{Chen1981,Chen2015}, by use of the celebrated Chern-Lashof inequality for the total absolute curvature. By using the flatness condition to analyze the second fundamental form and integrating the resulting curvature estimate over the unit normal sphere, he obtained a pointwise lower bound for the mean curvature, which yields the above estimate after integration. Since $2\pi^2>6\pi$, Chen's result in particular verifies the Kusner conjecture for flat Klein bottles. This motivates us to seek sharper lower bounds for the Willmore energy and to consider the Hirsch-M\"ader-Baumdicker conjecture in the flat case. By expanding an isometric immersion in terms of eigenfunctions of the flat Klein bottle and using the corresponding Fourier decomposition, we derive relations among the Fourier coefficients from the isometric condition. These relations, together with further estimates for the remaining terms, lead to explicit lower bounds for the Willmore energy and ultimately yield a uniform lower bound that proves the Hirsch-M\"ader-Baumdicker conjecture for flat immersions. 
\begin{theorem}\label{thm H-B-M iso}
Let $\psi_K:(\KB,ds_0^2)\to\rn$ be a smooth isometric immersion. Then \footnote{    Here $\mathrm{K}(\cdot)$ and $\mathrm{E}(\cdot)$ denote the complete elliptic
integrals of the first and second kinds respectively \cite{J-N-P}.}
\[\mathcal W(\psi_K)>2.136\pi^2\approx 6.710\pi>
 6\pi \mathrm E\!\left(\frac{2\sqrt2}{3}\right)\approx6.682\pi.\] 
In particular, the  Hirsch–M\"{a}der-Baumdicker conjecture holds strictly for flat Klein bottles in $\rn$.
\end{theorem}

We next construct a family of flat Klein bottles into various Euclidean spaces, with Willmore energy strictly less than $8\pi$.

Finally, by combining the Fourier expansion with the isometric condition and the minimal surface equation, we prove the nonexistence of flat minimal isometric Klein bottles in spheres.\\

\par {\bf {Plan of the paper.}}  We first recall the notions of conformal area and flat Klein bottles in Section \ref{section2}. In Section \ref{section4}, we derive lower bounds for the conformal area of Klein bottles and, consequently, for the Willmore energy of conformal immersions in the corresponding conformal classes. In Section \ref{section5}, we study the Willmore energy of isometrically flat Klein bottles. In Section \ref{section6}, we construct explicit isometric immersions of Klein bottles into Euclidean space with Willmore energy below $8\pi$ for a range of conformal classes. Finally, in Section \ref{section8}, we prove a nonexistence result for minimal isometric immersions of flat Klein bottles into spheres.

 \section{Conformal area and Flat Klein bottles }\label{section2}
 In this section, we will first recall the notion of conformal area of a surface \cite{L-Y}. Then we will collect some basic properties of Klein bottles. 
 
 \subsection{Conformal area of a closed surface}

\begin{definition}
     For a closed surface $M$ endowed with a fixed conformal structure, the conformal $N$-area $A_c(M,N)$ is defined as
 \begin{equation}
     A_c(M,N):= \underset{\psi}{\inf} \sup_{\mathcal{T}\in
    G} Area(\mathcal{T}\circ \psi),
 \end{equation}
 where $\psi$ runs over all branched conformal immersions of $M$ into  the $n$-dimensional unit sphere $\sn$. Here $G$ is the conformal group of $\sn$. The conformal area $A_c(M)$ is defined as
$A_c(M)=\inf_NA_c(M,N).$
 \end{definition}
 Let $\B^{N+1}\subset\mathbb{R}^{N+1}$ be the open unit ball. It is well known that there exists an embedding of $\B^{N+1}$ into $G$ (see e.g. \cite{L-Y}, \cite{M-N2} or \cite{M-R}) mapping $g\in\B^{N+1}$ into $\TT_g\in G$, with 
 $$\TT_g(x)=\frac{
x+
\left(
\left(
(1-|g|^2)^{-\frac12}-1
\right)
|g|^{-2}\langle x,g\rangle
+
(1-|g|^2)^{-\frac12}
\right)g
}{
(1-|g|^2)^{-\frac12}
\left(
\langle x,g\rangle+1
\right)
}, ~\hbox{ with }~ \TT_0=\mathrm{id} .$$
 
 Let $\dd s^2$ be a conformal metric on $M$ and we denote by $\dd M$ the associated measure. We have
 \begin{equation}
     A(\Tg\circ \psi):=Area(\Tg\circ \psi)=\frac{1}{2}\int_{M} \frac{1-|g|^2}{(\langle\psi,g\rangle+1)^2}|\nabla\psi|^2\, \dd M.
 \end{equation}
Let $\psi:M\to \sn\subset\mathbb R^{N+1}$ be an immersion. Denote by $\bar H$ and $H$ its mean curvature vectors in $\sn$ and $\mathbb R^{N+1}$, respectively. For immersions into the unit sphere, we get
\[
\W(\psi)
:=\int_M|H|^2\,\mathrm dM=\int_M\bigl(1+|\bar H|^2\bigr)\,\mathrm dM
=A(\psi)+\int_M|\bar H|^2\,\mathrm dM.
\]
Thus, this is precisely the Willmore energy of $\psi$ as an immersion into $\mathbb R^{N+1}$.
 
\begin{lemma}\label{lemma1.2}\cite{L-Y}
    Let $\psi:M\rightarrow\sn$ be a branched conformal immersion. Then there exists an element $g\in \B^{n+1}$ such that
\begin{equation}\label{eq1.3}
    \int_M \Tg\circ \psi \, \dd M=0.
\end{equation}
\end{lemma}
\begin{lemma}\cite{L-Y}
Let $\psi:M\rightarrow\rn$ be a  branched conformal immersion with mean curvature vector $H$. Then
    \begin{equation}\label{eq1.4}
         \int_M|H|^2\, \dd M\geq A_c(M).
    \end{equation}
\end{lemma}

\subsection{Flat Klein bottles and Fourier expansion}



 Consider the rectangular flat torus
\[\TB=\mathbb{C}/\Lambda,
\quad \hbox{ with }
\Lambda=\mathbb{Z}(1,0)\oplus \mathbb{Z}(0,b), \quad  b\in \R^+,\]
equipped with the flat metric
$\dd s_0^2=\dd x^2+\dd y^2=|\dd z|^2$, where $z=x+\mathrm{i}y$. It has area $Area(\TB)=b.$
The deck transformation  on $\TB$  is definedas (see for example \cite[Proposition 4.2.10]{BCGG})
\begin{equation}\label{eq-tau}
    \tau_b:z\mapsto \bar{z}+\frac{1}{2}+\Lambda.
\end{equation}
The quotient
$\KB:=\TB/\langle \tau_b\rangle$
gives a flat Klein bottle. We denote by
$\widetilde\pi:\TB\to \KB$
the covering map. Since $\tau_b$ is an isometric involution, the metric $\dd s_0^2$ descends to a flat metric on $\KB$, which by abuse of notation will still be denoted by $\dd s_0^2$. Then $\KB$ has area
\begin{equation}\label{eq-kb}
Area(\KB)=\frac{1}{2} Area(\TB)=\frac{b}{2}.
\end{equation}

 \begin{lemma}
     The moduli space of conformal Klein bottles can be identified with
$$
\mathcal M_{\mathrm{Klein}}\cong \mathbb R_+,
$$
where the point $b\in\mathbb R_+$ corresponds to the conformal class of the flat Klein bottle $\KB$.
 \end{lemma}
Throughout the paper, {\em $\KB$ will always denote the flat Klein bottle  with the flat metric $\dd s^2_0$.}  \\

Let $\psi_K:\KB\to \rn$ be a smooth map. Then the map
$\psi:=\psi_K\circ \widetilde\pi:\TB\to \rn$ is a lift on $\TB$ satisfying $\psi\circ \tau_b=\psi.$
Conversely, every smooth map $\psi:\TB\to \rn$ satisfying
$\psi\circ \tau_b=\psi$ descends uniquely to a smooth map $\psi_K:\KB\to \rn$.
 This observation will be used throughout the paper. Geometric quantities associated with $\psi_K$ can be computed on the covering torus in terms of $\psi$, subject to the symmetry induced by the involution $\tau_b$. In particular, if $\psi_K$ is an immersion, a conformal immersion, or an isometric immersion, then so is $\psi$ with respect to the lifted metric. Moreover, integrals over $\KB$ are equal to a half of the corresponding integrals over $\TB$.

Now consider the Fourier expansion adapted to the symmetry $\tau_b$. Since $\TB$ is flat, every $L^2$-function on $\TB$ admits a Fourier expansion with respect to $\Lambda$. 

Set
\begin{equation}\label{eq-S1S2}
\left\{\begin{split}
S_1&=\{(p,q)\in \mathbb{Z}\times \mathbb{Z}\mid  p\in2\mathbb{Z}^+\cup\{0\},\ q\in \mathbb{Z}^+\cup\{0\}\},\\
S_2&=\{(r,s)\in \mathbb{Z}\times \mathbb{Z}\mid  r\in2\mathbb{Z}^+-1,\ s\in \mathbb{Z}^+\}.
\end{split}\right.
\end{equation}
 Note that $S_1\cap S_2=\emptyset$. We define the functions
\begin{equation}\label{eq:eigenfunctions}
\left\{
\begin{aligned}
f_{pq}(x,y)&=\cos\left(2\pi px\right)\cos\left(\frac{2\pi qy}{b}\right), \\
g_{pq}(x,y)&=\sin\left(2\pi px\right)\cos\left(\frac{2\pi qy}{b}\right), \\
h_{rs}(x,y)&=\cos\left(2\pi rx\right)\sin\left(\frac{2\pi sy}{b}\right),\\
l_{rs}(x,y)&=\sin\left(2\pi rx\right)\sin\left(\frac{2\pi sy}{b}\right),
\end{aligned}
\right.
\end{equation}
on $\TB$,
where $(p,q)\in S_1$ and $(r,s)\in S_2$. Each of the nonzero functions in \eqref{eq:eigenfunctions} is an eigenfunction of the Laplace operator $\Delta_0=\partial_x^2+\partial_y^2$ on $(\TB,\dd s_0^2)$, with eigenvalue
\begin{equation}\label{eq-eigenvalue}
\lambda_{m,n}=4\pi^2(m^2+ 
\frac{n^2}{b^2}),~~ (m,n)=\hbox{$(p,q)\in S_1$ or $(r,s)\in S_2$}.
\end{equation}
Moreover, these functions are precisely adapted to the symmetry condition induced by $\tau_b$.  
 \begin{lemma}\label{lemma1.1}(see also \cite{B-H-K})
The nonzero functions 
\[\{f_{pq}\circ \tilde{\pi}^{-1}, g_{pq}\circ \tilde{\pi}^{-1}, h_{rs}\circ \tilde{\pi}^{-1}, l_{rs}\circ \tilde{\pi}^{-1}|(p,q)\in S_1, (r,s)\in S_2\}\] form a complete family of real eigenfunctions of the flat Klein bottle $\KB$.
\end{lemma}
For later calculations, it is convenient to normalize these eigenfunctions in $L^2(\TB)$. We denote by
$\bar f_{pq}$, $\bar g_{pq}$, $\bar h_{rs}$, $ \bar l_{rs}$
the corresponding normalized functions, chosen so that
\begin{equation}
    \label{eq-unit}
\int_{\TB} |\bar f_{pq}|^2\,\dd M_0
=
\int_{\TB} |\bar g_{pq}|^2\,\dd M_0
=
\int_{\TB} |\bar h_{rs}|^2\,\dd M_0
=
\int_{\TB} |\bar l_{rs}|^2\,\dd M_0
=1,\end{equation}
where $\dd M_0=\dd x\wedge\dd y$ denotes the area form of the flat metric $\dd s_0^2=\dd x^2+\dd y^2$. 
Although  $\bar g_{0q}$, $q>0$, does really show up, here we include them for convenience. 
Accordingly, any smooth map $\psi_K:\KB\to \rn$, identified with its $\tau_b$-invariant lift $\psi$ to $\TB$,  admits an expansion of the form
\begin{equation}\label{eq:fourier-expansion-map}
\psi
=
\sum_{(p,q)\in S_1}\bigl(A_{pq}\bar f_{pq}+B_{pq}\bar g_{pq}\bigr)
+
\sum_{(r,s)\in S_2}\bigl(A_{rs}\bar h_{rs}+B_{rs}\bar l_{rs}\bigr),
\end{equation}
for suitable vectors
$A_{pq},B_{pq},A_{rs},B_{rs}\in \rn.$ We also set 
\begin{equation}\label{eq-apq}
a_{pq}=|A_{pq}|^2+|B_{pq}|^2,\qquad (p,q)\in S_1,\qquad
a_{rs}=|A_{rs}|^2+|B_{rs}|^2,\qquad (r,s)\in S_2.
\end{equation}
We adopt the convention $B_{0q}=0$ for all $q\ge0$. Consequently, the definition in (\ref{eq-apq}) gives $a_{0q}=|A_{0q}|^2$. Note that the family
 $\{\bar f_{pq},\bar g_{pq},\bar h_{rs},\bar l_{rs}\}$
is orthogonal in $L^2(\TB)$, and the same remains true after applying $\partial_x$, $\partial_y$, or $\Delta$. \vspace{2mm}

In some case, it is also convenient to use the complex Fourier expansion. Let $\psi:(\T_b^2,\mathrm ds_0^2)\to\rn$ be a conformal immersion. Its complex Fourier expansion is
\begin{equation}
\psi(x,y)=\sum_{(m,n)\in\mathbb Z^2}C_{mn}e^{2\pi i\left(mx+\frac{ny}{b}\right)},
\label{eq:fourier-expansion}
\end{equation}
where $C_{mn}\in\mathbb C^n$. For $z,w\in\mathbb C^{n+1}$, we write $
\langle z,w\rangle_{\mathbb C}
:=z\cdot\overline w=z\overline{w}$. Since $\psi$ takes values in $\rn$,
\[
C_{-m,-n}=\overline{C_{mn}}.
\]
On the other hand, the deck transformation $\tau_b$ of the Klein bottle gives
\begin{equation}
C_{m,-n}=(-1)^mC_{mn}.
\label{eq:deck-relation}
\end{equation}
In particular, $C_{m0}=0$ whenever $m$ is odd. Thus, the lift $\psi$ of $\psi_K:(\KB,\mathrm{d}s_0^2)\to\rn$ admits the Fourier expansion \eqref{eq:fourier-expansion}, with coefficients satisfying \eqref{eq:deck-relation}.

\section{On the Willmore Energy of Klein bottles in $\sn$ }\label{section4}
In this section, we first apply the method of Montiel and Ros to estimate the area of balanced conformal immersions of a Klein bottle into spheres.

\begin{proposition}\label{prop:fourier-lower-bound}
Let $\psi_K:\KB\to \sn$ be a  branched conformal immersion of $(\KB,ds_0^2)$ into $\sn$. Assume that its lift $\psi$ satisfies
$\int_{\T_b^2}\psi\,\mathrm dM_0=0$. Then
\[
A(\psi_K)>\frac{8\pi^2b}{4b^2+1}.
\]
\end{proposition}

\begin{proof}
Set
\[
\mathbb Z^*:=\mathbb Z^2\setminus\left(\{(m,n):n=0,\ m\in2\mathbb Z+1\}\cup\{(0,0)\}\right).
\]
By \eqref{eq:fourier-expansion} and the assumption $\int_{\T_b^2}\psi\,\mathrm dM_0=0$, we have $C_{00}=0$. Differentiating the Fourier series gives
\[
\begin{split}
\psi_x
&=2\pi \mathrm i\sum_{(m,n)\in\mathbb Z^*}mC_{mn}\mathrm e^{2\pi i\left(mx+\frac{ny}{b}\right)},\quad \psi_y
=\frac{2\pi \mathrm i}{b}\sum_{(m,n)\in\mathbb Z^*}nC_{mn}\mathrm e^{2\pi i\left(mx+\frac{ny}{b}\right)}.
\end{split}
\]
Hence
\[
\begin{split}
\int_{\T_b^2}|\psi|^2\,\mathrm dM_0
&=b\sum_{(m,n)\in\mathbb Z^*}|C_{mn}|^2,\\
\int_{\T_b^2}|\psi_x|^2\,\mathrm dM_0
&=4\pi^2b\sum_{(m,n)\in\mathbb Z^*}m^2|C_{mn}|^2,\\
\int_{\T_b^2}|\psi_y|^2\,\mathrm dM_0
&=\frac{4\pi^2}{b}\sum_{(m,n)\in\mathbb Z^*}n^2|C_{mn}|^2.
\end{split}
\]
Since $\psi$ is a conformal immersion into $\sn$, we obtain
$
|\psi|^2=1,$ $ |\psi_x|^2=|\psi_y|^2,$ $\langle\psi_x,\psi_y\rangle=0.
$
Therefore
\begin{equation}
\sum_{(m,n)\in\mathbb Z^*}(b^2m^2-n^2)|C_{mn}|^2=0,
\label{eq:prop31-conformal}
\end{equation}
\begin{equation}
\sum_{(m,n)\in\mathbb Z^*}|C_{mn}|^2=1,
\label{eq:prop31-mass}
\end{equation}
\begin{equation}
A(\psi_K)
:=\frac14\int_{\T_b^2}|\nabla\psi|^2\,\mathrm dM_0
=\frac12\int_{\T_b^2}|\psi_x|^2\,\mathrm dM_0
=2\pi^2b\sum_{(m,n)\in\mathbb Z^*}m^2|C_{mn}|^2.
\label{eq:prop31-area}
\end{equation}
By \eqref{eq:prop31-conformal} and \eqref{eq:prop31-mass}, we get
\[
\begin{split}
\sum_{(m,n)\in\mathbb Z^*}\bigl[m^2+4(n^2-1)\bigr]|C_{mn}|^2
&=\sum_{(m,n)\in\mathbb Z^*}m^2|C_{mn}|^2
+4\sum_{(m,n)\in\mathbb Z^*}n^2|C_{mn}|^2
-4\sum_{(m,n)\in\mathbb Z^*}|C_{mn}|^2\\
&=(4b^2+1)\sum_{(m,n)\in\mathbb Z^*}m^2|C_{mn}|^2-4.
\end{split}
\]
Since $m^2+4(n^2-1)\ge0$ for every $(m,n)\in\mathbb Z^*$, it follows that
$
\sum_{(m,n)\in\mathbb Z^*}m^2|C_{mn}|^2\ge\frac4{4b^2+1}.
$
Together with \eqref{eq:prop31-area}, this yields
\[
A(\psi_K)\ge\frac{8\pi^2b}{4b^2+1}.
\]

Consider the equality case. If
$A(\psi_K)=\frac{8\pi^2b}{4b^2+1}$, only the modes $(0,\pm1)$ and $(\pm2,0)$ remain. A simple calculation shows this is impossible.
\end{proof}

The proof of Proposition \ref{prop:fourier-lower-bound} shows that
\[
R:=\sum_{(m,n)\in\mathbb Z^*}\bigl[m^2+4(n^2-1)\bigr]|C_{mn}|^2>0.
\]
Next we derive a quantitative lower bound of $R$, which improves the estimate for $A(\psi_K)$.

\begin{theorem}\label{thm:improved-lower-bound}
Let $\psi_K:\KB\to \sn$ be a branched conformal immersion of $(\KB,ds_0^2)$ into $\sn$. Assume that its lift $\psi$ satisfies
$\int_{\T_b^2}\psi\,\mathrm  dM_0=0$. Then
\[
A(\psi_K)>
\frac{\pi^2b}{4b^2+1}\bigl(8+2R_*(b)\bigr),
\]
where
\[
R_*(b)=
\frac{1024b^4}
{576b^6+3056b^4+828b^2+9
+(4b^2+1)(4b^2+3)\sqrt{1808b^4+1560b^2+9}}.
\]
\end{theorem}

To prove Theorem \ref{thm:improved-lower-bound}, we need several lemmas. First set
\[
\mathfrak u:=|C_{01}|^2+|C_{0,-1}|^2=2|C_{01}|^2,
\qquad
\mathfrak v:=|C_{20}|^2+|C_{-2,0}|^2=2|C_{20}|^2.
\]
By Proposition \ref{prop:fourier-lower-bound}, we have $R>0$, and hence
\begin{equation}
\mathfrak u+\mathfrak v<1.
\label{eq:uvR}
\end{equation}

To derive an inequality involving only $R$, we seek two estimates relating $\mathfrak u$, $\mathfrak v$, and $R$. Combining them with \eqref{eq:uvR} will then yield an inequality for $R$. To this end, we use the conformal condition and $|\psi|=1$ to derive lower and upper bounds for $\mathfrak u$ in terms of $\mathfrak v$ and $R$.

\begin{lemma}\label{lem:u-lower}
The following lower bound for $\mathfrak u$ holds:
\begin{equation}
\mathfrak u\ge\frac{4b^2}{4b^2+1}-\frac{1+3b^2}{4b^2+1}R.
\label{eq:lemma33-lower}
\end{equation}
\end{lemma}

\begin{proof}
Let
\[
w_{mn}:=m^2+4(n^2-1),
\qquad
\widehat\Sigma:=\mathbb Z^*\setminus\{(0,\pm1),(\pm2,0)\}.
\]
By Proposition \ref{prop:fourier-lower-bound}, we have
\begin{equation}
\mathfrak u+\mathfrak v+\sum_{(m,n)\in\widehat\Sigma}|C_{mn}|^2=1,
\label{eq:lemma33-mass}
\end{equation}
and \eqref{eq:prop31-conformal} gives
\begin{equation}
\mathfrak u=4b^2\mathfrak v+\sum_{(m,n)\in\widehat\Sigma}(b^2m^2-n^2)|C_{mn}|^2.
\label{eq:lemma33-conformal}
\end{equation}
Substituting \eqref{eq:lemma33-mass} into \eqref{eq:lemma33-conformal}, we obtain
\[
(4b^2+1)\mathfrak u
=4b^2+\sum_{(m,n)\in\widehat\Sigma}
\bigl[b^2(m^2-4)-n^2\bigr]|C_{mn}|^2.
\]
For every $(m,n)\in\widehat\Sigma$, it is easy to see that $b^2(m^2-4)-n^2+(1+3b^2)\bigl[m^2+4(n^2-1)\bigr]=(4b^2+1)(m^2+3n^2-4)\ge0$. Hence, by the definition of $R$, we have
\[
(4b^2+1)\mathfrak u\ge4b^2-(1+3b^2)R.
\]
Dividing by $4b^2+1$ gives \eqref{eq:lemma33-lower}.
\end{proof}

We next derive an upper bound for $\mathfrak u$ in terms of $\mathfrak v$ and $R$.

\begin{proposition}\label{lem:u-upper}
The following upper bound for $\mathfrak u$ holds:
\begin{equation}
\mathfrak u\le\frac{4b^2+3}{4\sqrt2}\sqrt{\mathfrak vR}+\frac{4+b^2}{2}R.
\label{eq:lemma34-upper}
\end{equation}
\end{proposition}

\begin{proof}
By Lemma \ref{lemma-id-u}, we get
$$\frac {\mathfrak u}{2}
+(4b^2+3)\operatorname{Re}
\bigl(C_{2,0}\overline{C_{2,-2}}\bigr)
+\mathcal J=0.$$
Since $w_{2,2}=w_{2,-2}=w_{-2,2}=w_{-2,-2}=16$, the definition of $R$ gives
\begin{equation}
4\times16|C_{2,-2}|^2=64|C_{2,-2}|^2\le R.
\label{eq:lemma34-C22}
\end{equation}
Moreover, by the definition of $v$, we have
\begin{equation}
|C_{2,0}|^2=\frac {\mathfrak v}{2}.
\label{eq:lemma34-C20}
\end{equation}
Combining \eqref{eq:lemma34-C22} and \eqref{eq:lemma34-C20},
\begin{equation}
\left|
\operatorname{Re}\bigl(C_{2,0}\overline{C_{2,-2}}\bigr)
\right|
\le |C_{2,0}\overline{C_{2,-2}}|
\le\sqrt{\frac {\mathfrak v}{2}}\sqrt{\frac R{64}}
=\frac{\sqrt{\mathfrak vR}}{8\sqrt2}.
\label{eq:lemma34-cross}
\end{equation}
By Lemma \ref{lem:A1},
we get $|\mathcal J|\le\frac{4+b^2}{4}R$.
Combining this with \eqref{eq:lemma34-decomposition} and
\eqref{eq:lemma34-cross} yields  \eqref{eq:lemma34-upper}.
\end{proof}

We now combine Lemma \ref{lem:u-lower} and Proposition \ref{lem:u-upper} to derive a positive lower bound for $R$.

\begin{lemma}\label{lem:R-lower}
The following lower bound for $R$ holds:
\[
R>\frac{32b^4}{(4b^2+1)^2(\sqrt{\mathfrak B}-\nu)},
\]
where $
\mathfrak B:=\nu^2-4\mu\frac{16b^4}{(4b^2+1)^2}$, and $\mu,\nu$ are defined by
\begin{equation}
\begin{split}
\mu
&:=\left(
\frac{1+3b^2}{4b^2+1}
+\frac{4+b^2}{2}
\right)^2
-\frac{(4b^2+3)^2(1+3b^2)}
{32(4b^2+1)},\\
\nu
&:=-\frac{8b^2}{4b^2+1}
\left(
\frac{1+3b^2}{4b^2+1}
+\frac{4+b^2}{2}
\right)
+\frac{(4b^2+3)^2}{32}
\left(
\frac{4b^2}{4b^2+1}-1
\right).
\end{split}
\label{eq:mu-nu}
\end{equation}
\end{lemma}

\begin{proof}
Substituting \eqref{eq:uvR} into \eqref{eq:lemma34-upper}, we obtain
\[
\mathfrak u<
\frac{4b^2+3}{4\sqrt2}\sqrt{(1-\mathfrak u)R}
+\frac{4+b^2}{2}R.
\]
Combining this with \eqref{eq:lemma33-lower}, we get
\begin{equation}
\begin{split}
\frac{4b^2}{4b^2+1}
<&\frac{1+3b^2}{4b^2+1}R
+\frac{4b^2+3}{4\sqrt2}
\sqrt{
\left(
1-\frac{4b^2}{4b^2+1}
+\frac{1+3b^2}{4b^2+1}R
\right)R
}+\frac{4+b^2}{2}R.
\end{split}
\label{eq:R-inequality}
\end{equation}
Define
\[
L(R):=
\left(
\frac{1+3b^2}{4b^2+1}
+\frac{4+b^2}{2}
\right)R
+\frac{4b^2+3}{4\sqrt2}
\sqrt{
\left(
1-\frac{4b^2}{4b^2+1}
+\frac{1+3b^2}{4b^2+1}R
\right)R
},
\qquad R>0.
\]
Since $\frac{1+3b^2}{4b^2+1}+\frac{4+b^2}{2}>0$ and $1-\frac{4b^2}{4b^2+1}>0$, $L(R)$ is strictly increasing on $(0,\infty)$ and $L(R)\to0$ as $R\to0$. Therefore there exists a unique
$R_*>0$ such that
\[
L(R_*)=\frac{4b^2}{4b^2+1}.
\]
By \eqref{eq:R-inequality}, $R>R_*>0$. It remains to solve for $R_*$. By the definition of $R_*$, we have
\[
\frac{4b^2}{4b^2+1}
=
\left(
\frac{1+3b^2}{4b^2+1}
+\frac{4+b^2}{2}
\right)R_*
+\frac{4b^2+3}{4\sqrt2}
\sqrt{
\left(
1-\frac{4b^2}{4b^2+1}
+\frac{1+3b^2}{4b^2+1}R_*
\right)R_*
}.
\]
Rearranging and squaring gives
\[
\left[
\frac{4b^2}{4b^2+1}
-
\left(
\frac{1+3b^2}{4b^2+1}
+\frac{4+b^2}{2}
\right)R_*
\right]^2
=
\frac{(4b^2+3)^2}{32}
\left(
1-\frac{4b^2}{4b^2+1}
+\frac{1+3b^2}{4b^2+1}R_*
\right)R_*.
\]
Equivalently,
\begin{equation}
\mu R_*^2+\nu R_*+\frac{16b^4}{(4b^2+1)^2}=0.
\label{eq:R-quadratic}
\end{equation}
Using \eqref{eq:mu-nu} and \eqref{eq:R-quadratic}, we get
\[
\mathfrak B
=\nu^2-4\mu\frac{16b^4}{(4b^2+1)^2}
=\nu^2-4\mu(-\mu R_*^2-\nu R_*)
=(2\mu R_*+\nu)^2.
\]
Moreover, it follows that
\[
\begin{split}
-(2\mu R_*+\nu)
=&-2R_*
\left(
\left(
\frac{1+3b^2}{4b^2+1}
+\frac{4+b^2}{2}
\right)^2
-\frac{(4b^2+3)^2(1+3b^2)}
{32(4b^2+1)}
\right)\\
&\qquad+\frac{8b^2}{4b^2+1}
\left(
\frac{1+3b^2}{4b^2+1}
+\frac{4+b^2}{2}
\right)
+\frac{(4b^2+3)^2}{32}
\left(
1-\frac{4b^2}{4b^2+1}
\right)\\
=&2
\left(
\frac{1+3b^2}{4b^2+1}
+\frac{4+b^2}{2}
\right)
\left(
\frac{4b^2}{4b^2+1}
-
\left(
\frac{1+3b^2}{4b^2+1}
+\frac{4+b^2}{2}
\right)R_*
\right)\\
&\qquad+\frac{(4b^2+3)^2}{32}
\left(
1-\frac{4b^2}{4b^2+1}
+2\frac{1+3b^2}{4b^2+1}R_*
\right)\\
>{}&0.
\end{split}
\]
Thus, $2\mu R_*+\nu=-\sqrt{\mathfrak B}$. Combining this with \eqref{eq:R-quadratic} gives
\[
R_*=
\frac{32b^4}
{(4b^2+1)^2(\sqrt{\mathfrak B}-\nu)}.
\]
Since $R>R_*$, the lemma follows.
\end{proof}
The lower bound in Lemma \ref{lem:R-lower} can be written explicitly in terms of $b$.

\begin{proof}[Proof of Theorem \ref{thm:improved-lower-bound}]
By Lemma \ref{lem:R-lower}, we have
$R>R_*(b)$. A direct computation gives
\[
R_*(b)=
\frac{1024b^4}
{576b^6+3056b^4+828b^2+9
+(4b^2+1)(4b^2+3)
\sqrt{1808b^4+1560b^2+9}}.
\] 
Since
$
R=(4b^2+1)
\sum_{(m,n)\in\mathbb Z^*}m^2|C_{mn}|^2-4,
$ 
\eqref{eq:prop31-area} yields Theorem \ref{thm:improved-lower-bound}.
\end{proof}
As an immediate consequence of Theorem \ref{thm:improved-lower-bound}, we obtain the following corollary.
\begin{corollary}\label{cor3.2}
    For every Klein bottle $\KB$, we have
    $$A_c(\KB)\geq \mathcal{R}(b):=\frac{\pi^2b}{4b^2+1}\bigl(8+2R_*(b)\bigr).$$
    \end{corollary}
    \begin{proof}
        By Lemma \ref{lemma1.2},  for any  branched conformal immersion $\psi_K:\KB\rightarrow\sn$ with a lift $\psi$, there exists $g \in \B^{N+1}$ such that$ \int_{\TB} \Tg\circ \psi \, \dd M_0=0$. 
    By  Theorem \ref{thm:improved-lower-bound}, we obtain 
       $Area(\Tg\circ \psi_K)>\frac{\pi^2b}{4b^2+1}\bigl(8+2R_*(b)\bigr). $
In particular, $\sup_{g\in\B^{n+1}}A(\Tg\circ \psi_K)\geq\frac{\pi^2b}{4b^2+1}\bigl(8+2R_*(b)\bigr)$.  Hence   $A_c(\KB)\geq\frac{\pi^2b}{4b^2+1}\bigl(8+2R_*(b)\bigr)$.
    \end{proof}
\begin{proof}[Proof of Theorem \ref{thm conf imm}]
 It follows from (\ref{eq1.4}) and Corollary \ref{cor3.2}.
\end{proof}
    \begin{corollary}
     Let $\psi_K:\KB\longrightarrow \rn$ be a branched conformal immersion. 
Let $b_3<b_4$ be the two solutions to
$\mathcal{R}(b)=\frac{\pi^2b}{4b^2+1}\bigl(8+2R_*(b)\bigr)=2\pi^2$. 
Numerically, $b_3\approx0.408$ and $b_4\approx0.648$. If  $b\in\left(b_3,b_4\right)$, then
    $$ \mathcal{W}(\psi_K)>2\pi^2\approx 6.283\pi.$$

\end{corollary}

\begin{remark}~

\begin{enumerate}
\item The bound $\mathcal{R}(b)$ in Corollary~\ref{cor3.2} takes maximum at $\hat{b}\approx0.515$ with $\mathcal{R}(\hat b) \approx20.278
\approx6.454\,\pi$.
    \item In \cite{H-M}, Hirsch and M\"ader-Baumdicker conjectured that for all immersions $\psi$ from Klein bottle into $\rn(n\geq 4)$, its Willmore energy satisfies
    $$\mathcal{W}(\psi)\geq 6\pi \mathrm{E}(\frac{2\sqrt{2}}{3})\approx6.682\pi,$$
   with equality holding if and only if $\psi$ is M\"{o}bius congruent to  Lawson's bipolar minimal Klein bottle $\tilde\tau_{3,1}:\mathbb{K}^2_{b_0}\rightarrow \mathbb S^4$ \cite{Lawson-1970}, where $b_0=\frac{2\mathrm{K}(1/2)}{2\pi}\approx0.537$. Moreover, they proved this conjecture when $\KB=\mathbb{K}^2_{b_0}$. Note that the above estimate does not yield  any progress toward this conjecture. 
\end{enumerate}
 
\end{remark}

\section{Estimates of Willmore Energy for isometrically flat Klein bottles}\label{section5}
In this section, we first estimate the Willmore energy of flat Klein bottles in $\rn$ by Li-Yau's method \cite{L-Y}. We next use the eigenfunction expansion on the Klein bottle, together with the isometric condition, to derive a sharper lower bound for the Willmore energy. This yields the following theorem.
\begin{theorem}\label{thm-iso}
Let $\psi_K:(\KB,\mathrm ds_0^2)\to\rn$ be an isometric immersion, then
\begin{equation}
    \label{eq-main}
    \mathcal W(\psi_K)
>
\begin{cases}
\displaystyle
\frac{\pi^2}{2}\left(b+\frac{3}{b}\right)
+
\frac{2\pi\DD}{b}
\tanh\left(
\frac{\pi\DD}{2}
\right),
& b>\sqrt{\frac{2}{3}},\\[10pt]
\displaystyle
\frac{11}{12}\sqrt{6}\,\pi^2,
& b=\sqrt{\frac{2}{3}},\\[10pt]
\displaystyle
\frac{\pi^2}{2}\left(4b+\frac{1}{b}\right)
+
\frac{\pi\BB}{b}
\cot\left(
\frac{\pi\BB}{2}
\right),
& 0<b<\sqrt{\frac{2}{3}}.
\end{cases}
\end{equation}
Here $\BB:=\sqrt{1-2b^2+3b^4}$ and $\DD:=\sqrt{\frac{3b^2-2}{2}}$.
\end{theorem}

We separate the proof into several propositions.

 \subsection{A first estimate via the inequality of Li-Yau}
 
\begin{proposition}
    \label{theorem2.2.1}
   Let $\psi_K:(\KB,\dd s^2_0)\rightarrow\rn$ be an isometric immersion. Then we have
   \begin{equation}\label{ineq-LY}
   \mathcal{W}(\psi_K)\geq
   \begin{cases}
     \frac{\pi^2}{2}(b+\frac{3}{b}),&\text{if} \quad b>\sqrt{\frac{2}{3}},\\
       \frac{\pi^2}{2}(4b+\frac{1}{b}),&\text{if} \quad b\leq \sqrt{\frac{2}{3}},
   \end{cases}
   \end{equation}
\end{proposition}
\begin{proof}
We follow the method of Li and Yau in \cite{L-Y},  replacing the torus with the Klein bottle. Since $\psi_K:\KB\longrightarrow\rn$ is isometric, we get as in \cite{L-Y} that
\begin{equation}\label{eq2.2.1}
 \mathcal{W}(\psi_K)=\frac{1}{2}\int_{\TB}|H|^2\dd M_0=\frac{1}{8}\int_{\TB}|\Delta \psi|^2\dd M_0,
\end{equation}
 where $\psi$ is its lift to $\TB$ and $H$ is the mean curvature vector of $\psi$ in $\mathbb{R}^{n}$.
 By translation, set
\begin{equation}\label{eq2.2.2}
    \int_{\TB}\psi \;\dd M_0=0.
\end{equation}
Set $S^*_1:=S_1-\{(0,0)\}$. By the Fourier expansion \eqref{eq:fourier-expansion-map} and \eqref{eq-apq} of $\psi$, we get
 \begin{equation*}
     \begin{aligned}
        \frac{1}{2} \int_{\TB}|\Delta\psi|^2\dd M_0&=8\pi^4 \left(\mathop{\sum}\limits_{(p,q)\in S^*_1}(p^2+\frac{q^2}{b^2})^2a_{pq}+\mathop{\sum}\limits_{(r,s)\in S_2}(r^2+\frac{s^2}{b^2})^2a_{rs}\right)\\
        &=8\pi^4 \left(\mathop{\sum}\limits_{(p,q)\in S^*_1}(p^4+\frac{q^4}{b^4}+2\frac{p^2q^2}{b^2})a_{pq}+\mathop{\sum}\limits_{(r,s)\in S_2}(r^4+\frac{s^4}{b^4}+2\frac{r^2s^2}{b^2})a_{rs}\right).\\
     \end{aligned}
 \end{equation*}
It is direct to check that
  \begin{equation*}
      \begin{split}
          p^4+\frac{q^4}{b^4}+2\frac{p^2q^2}{b^2}&\geq 4p^2+\frac{q^2}{b^4}, \qquad\qquad  \text{~for}\quad (p,q)\in  S^*_1, \\
          r^4+\frac{s^4}{b^4}+2\frac{r^2s^2}{b^2}&\geq r^2+\frac{s^2}{b^4}+2\frac{r^2}{b^2},\qquad \text{for}\quad(r,s)\in S_2.
      \end{split}
  \end{equation*}
  As a result,
  \begin{equation*}
       \frac{1}{2} \int_{\TB}|\Delta\psi|^2\dd M_0\geq8\pi^4 \left(\mathop{\sum}\limits_{(p,q)\in S^*_1}(4p^2+\frac{q^2}{b^4})a_{pq}+\mathop{\sum}\limits_{(r,s)\in S_2}(r^2+\frac{s^2}{b^4}+2\frac{r^2}{b^2})a_{rs}\right).
  \end{equation*}
Setting $ S^*= S^*_1\cup S_2$, since
\[4>1+\frac{2}{b^2} \hbox{ when }b>\sqrt{\frac{2}{3}}, \hbox{ and }4\leq 1+\frac{2}{b^2} \hbox{ when } b\leq\sqrt{\frac{2}{3}},\]
 we obtain
 \begin{equation*}
 \frac{1}{2}\int_{\TB}|\Delta \psi|^2\dd M_0\geq
 \begin{cases}
     8\pi^4 \left(\mathop{\sum}\limits_{(k,t)\in S^*}\left(\left(1+\frac{2}{b^2}\right)k^2+\frac{t^2}{b^4}\right)a_{kt}\right),&\text{if} \quad b>\sqrt{\frac{2}{3}},\\
    8\pi^4 \left(\mathop{\sum}\limits_{(k,t)\in S^*}\left(4k^2+\frac{t^2}{b^4}\right)a_{kt}\right),&\text{if} \quad b\leq \sqrt{\frac{2}{3}}.
 \end{cases}
 \end{equation*}

Since $\psi:\TB\longrightarrow\rn$ is isometric,   we obtain
 \begin{equation*}
     \begin{aligned}
        2b=2A(\TB)=\int_{\TB}|\nabla \psi|^2\dd M_0&=2\int_{\TB}|\psi_x|^2\dd M_0=8\pi^2 \mathop{\sum}\limits_{(k,t)\in S^*}k^2a_{kt},\\
                2b=2A(\TB)=\int_{\TB}|\nabla \psi|^2\dd M_0&=2\int_{\TB}|\psi_y|^2\dd M_0=8\pi^2 \mathop{\sum}\limits_{(k,t)\in S^*}\frac{t^2}{b^2}a_{kt}.
     \end{aligned}
 \end{equation*}
Hence
 \begin{equation*}
  \frac{1}{2}\int_{\TB}|\Delta \psi|^2\dd M_0\geq
 \begin{cases}
   4\pi^2A(\KB)(1+\frac{2}{b^2}+\frac{1}{b^2})=2\pi^2(b+\frac{3}{b}),&\text{if} \quad b>\sqrt{\frac{2}{3}},\\
    4\pi^2A(\KB)(4+\frac{1}{b^2})=2\pi^2(4b+\frac{1}{b}),&\text{if} \quad b\leq \sqrt{\frac{2}{3}}.
 \end{cases}
 \end{equation*}
 Together with \eqref{eq2.2.1}, we finish the proof.
\end{proof}

\subsection{An further improvement}
Set
\begin{equation*}
d_{kt}
=
\left(k^2+\frac{t^2}{b^2}\right)^2
-
\left(1+\frac{2}{b^2}\right)k^2
-
\frac{t^2}{b^4},
\qquad
\widetilde d_{kt}
=
\left(k^2+\frac{t^2}{b^2}\right)^2
-
4k^2
-
\frac{t^2}{b^4},
\end{equation*}
and
\begin{equation*}
D
=
\sum_{(p,q)\in S_1^*}d_{pq}a_{pq}
+
\sum_{(r,s)\in S_2}d_{rs}a_{rs},
\qquad
\widetilde D
=
\sum_{(p,q)\in S_1^*} \widetilde d_{pq}a_{pq}
+
\sum_{(r,s)\in S_2} \widetilde d_{rs}a_{rs}.
\end{equation*}
Then by Proposition \ref{theorem2.2.1}, we have
\begin{equation}
d_{kt}\ge 0
\quad\text{when } b>\sqrt{\frac{2}{3}},
\qquad
\widetilde d_{kt}\ge 0
\quad\text{when } 0<b\le\sqrt{\frac{2}{3}},\qquad \forall(k,t)\in S^*.
\label{eq:d-positive}
\end{equation}
Using the Fourier identity in the proof of Proposition \ref{theorem2.2.1}, we obtain
\begin{equation}\label{eq-W-est}
\mathcal W(\psi_K)
=
\begin{cases}
\displaystyle
\frac{\pi^2}{2}\left(b+\frac{3}{b}\right)+2\pi^4D,
& b>\sqrt{\frac{2}{3}},\\[6pt]
\displaystyle
\frac{\pi^2}{2}\left(4b+\frac{1}{b}\right)+2\pi^4\widetilde D,
& 0<b\le \sqrt{\frac{2}{3}}.
\end{cases}
\end{equation}
If the equality holds in (\ref{ineq-LY}), we obtain
\begin{equation*}
\mathcal A^*(\psi)
\subset
\{(0,1),(1,1)\},
\qquad
b>\sqrt{\frac{2}{3}},
\end{equation*}
or
\begin{equation*}
\mathcal A^*(\psi)
\subset
\{(0,1),(2,0)\},
\qquad
0<b<\sqrt{\frac{2}{3}},
\end{equation*}
where $\mathcal A^*(\psi):=\{(k,t)\in S^*\mid a_{kt}\neq 0\}$. In particular, $\mathcal A^*(\psi)
\subset\{(0,1),(2,0),(1,1)\}$, for $b=\sqrt{\frac{2}{3}}$.

An argument similar to that in the proof of Proposition \ref{prop:fourier-lower-bound} shows that equality in (\ref{ineq-LY}) cannot occur.  Consequently, the inequality in (\ref{ineq-LY}) is strict. Thus, the problem reduces to obtaining strictly positive lower bounds for $D$ and $\widetilde D$, which is addressed in the following propositions.
\begin{lemma}\label{lem-piecewise-lower-bound}
Let $
\psi_K:(\KB,\mathrm ds_0^2)\to\rn $ be an isometric immersion. 

\begin{enumerate}

\item If $
b>\sqrt{\frac23}$,
then $
\mathcal W(\psi_K)
>
\frac{\pi^2}{2}
\left(
b+\frac3b
\right)
+
\frac{2\pi\DD}{b}
\tanh\left(
\frac{\pi\DD}{2}
\right)$.
\item If $
0<b<\sqrt{\frac23}$, then $
\mathcal W(\psi_K)
>
\frac{\pi^2}{2}
\left(
4b+\frac1b
\right)
+
\frac{\pi\BB}{b}
\cot\left(
\frac{\pi\BB}{2}
\right)$.
\end{enumerate}
\end{lemma}
 The details of the proofs of Lemma \ref{lem-piecewise-lower-bound} is given in Appendix \ref{appendix-proof}.  Here we outline a proof of Lemma \ref{lem-piecewise-lower-bound}. 
When $b>\sqrt{\frac{2}{3}}$, we first evaluate $\psi_x$ at $y=0$ and $y=\frac{b}{2}$, which allows us to eliminate the contribution of the Fourier modes $(r,s)\in S_2$. We then consider
$\int_0^1|\psi_x(x,0)|^2\,\mathrm dx$ and
$\int_0^1\left|\psi_x\left(x,\frac{b}{2}\right)\right|^2\,\mathrm dx$.
This yields the identity
\begin{equation}
\frac{1}{2\pi^2}
=
\sum_{p\in2\mathbb Z^+}
p^2
\left(
|A_p^e|^2
+
|A_p^o|^2
+
|B_p^e|^2
+
|B_p^o|^2
\right).
\label{eq:sum-Ap}
\end{equation}
Here $A_p^e$, $A_p^o$, $B_p^e$, and $B_p^o$ are defined in \eqref{eq:Ape-definition}. Next, we estimate $d_{pq}$ pointwise and apply the Cauchy-Schwarz inequality to obtain upper bounds for
$p^2|A_p^e|^2$, $p^2|A_p^o|^2$, $p^2|B_p^e|^2$ and $p^2|B_p^o|^2$.
Combining these estimates with \eqref{eq:sum-Ap}, we obtain a lower bound for $D$. Finally, together with \eqref{eq-W-est}, this gives a lower bound for $\mathcal W(\psi_K)$. When $0<b<\sqrt{\frac23}$, the argument is the same as in the case $b>\sqrt{\frac23}$, with differentiation with respect to $y$ in place of $x$.
\begin{proof}[Proof of Theorem \ref{thm-iso}]
    Theorem \ref{thm-iso} follows from Propositions \ref{theorem2.2.1} and Lemma \ref{lem-piecewise-lower-bound}.
\end{proof}

\subsection{Applications to the Hirsch-M\"ader-Baumdicker conjecture}

\begin{proof}[Proof of Theorem \ref{thm H-B-M iso}]

We first define the following two quantities:
\begin{equation}\label{eq-J1}
    J_1(b)
:=
\frac{\pi^2}{2}
\left(
b+\frac3b
\right)
+
\frac{2\pi\DD}{b}
\tanh\left(
\frac{\pi\DD}{2}
\right),
\end{equation}
\begin{equation}\label{eq-J2}
    J_2(b)
:=
\frac{\pi^2}{2}\left(4b+\frac{1}{b}\right)
+
\frac{\pi\BB}{b}
\cot\left(
\frac{\pi\BB}{2}
\right).
\end{equation}
Next set $b_*\in\left(\frac12,\frac{1}{\sqrt3}\right)$ to be the solution to (Numerically 
$b_*\approx0.513$)
    \begin{equation}\label{eq-b*}
2-\frac{1}{2b_*^2}
+
\frac{3b_*^4-1}
{  2b_*^2\tilde{b}_*}
\cot\tilde{b}_*
-
(3b_*^2-1)
\csc^2\tilde{b}_*
=0, \hbox{   with $\tilde{b}_*:=\frac{\pi}{2}\sqrt{1-2b_*^2+3b_*^4}$.}
    \end{equation}
 By Lemma \ref{lemma-b-big}, we get
\[
J_1(b)>\inf_{b>\sqrt{\frac{2}{3}}} J_1(b)
=
\lim_{b\to\sqrt{\frac{2}{3}}} J_1(b)
=
\frac{11}{12}\sqrt{6}\,\pi^2
\approx
2.245\,\pi^2.
\]
Lemma \ref{lemma-b-small} and Lemma \ref{lem:J2-lower-bound} yield
\[
\min_{0<b<\sqrt{\frac{2}{3}}}J_2(b)
=
J_2(b_*)
>\frac{\pi^2}{2}\sqrt{5+4\sqrt{11}}
\approx
6.713\pi>2.136\,\pi^2.
\]
Theorem \ref{thm H-B-M iso} then follows directly.
\end{proof}

\section{Examples of isometric immersions from $\KB$ into $\rn$}\label{section6}

We first recall  Tompkins' classical example \cite{T-1941}.
 \begin{example}
     Tompkins' example of Klein bottle into $\mathbb{R}^4$ in \cite{T-1941} is of the form
$$\psi_T=\left(\cos u\cos v,\sin u \cos v, 2\cos\frac{u}{2}\sin v,2\sin\frac{u}{2}\sin v\right).$$
Writing as a conformal immersion   $\psi_T:\mathbb{K}^2_{b_T}\rightarrow\R^4$, we can get that
\begin{equation}\label{eq-Tompkins}
        b_T=\frac{2}{\pi}\mathrm{E}\left(\frac{\sqrt{3}}{2}\right)\approx0.771, \hbox{ and }
     \mathcal{W}(\psi_T)=\frac{4\pi}{3}\left(5\mathrm{E}\left(\frac{\sqrt{3}}{2}\right)+\mathrm{K}\left(\frac{\sqrt{3}}{2}\right)\right)\approx10.95\pi.
\end{equation}
 \end{example}
 The above example shows that explicit immersions can be constructed in certain conformal classes, although their Willmore energy may still be larger than $8\pi$. 
 Breuning, Hirsch and Mäder-Baumdicker proved in \cite{B-H-M} that every conformal class of Klein bottles admits an embedding into $\mathbb{R}^4$ with Willmore energy equal to $8\pi$. Hence, a natural question is whether one can construct flat immersions in various conformal class whose Willmore energy is strictly below $8\pi$. In this section, we construct examples of isometric immersions of $\KB$ into $\rn$ from three different perspectives: further discussions on the inequality of Li-Yau in Proposition \ref{theorem2.2.1}; deformations of Tompkins' Klein bottle;  equivariant examples via odd and even functions.

\subsection{Construction from discrete Fourier modes}

Motivated by the low-frequency modes $(0,1)$, $(2,0)$, and $(1,1)$ appearing in the proof of Proposition \ref{theorem2.2.1} we construct explicit isometric immersions by adding further modes of relatively small norm. 
\begin{example}\label{exmaple b<1}
  For  $0<b\leq1$,  let   $\psi_b:\mathbb{K}^2_b\to \mathbb{S}^6\left(\frac{\sqrt{1+3b^2}}{4\pi}\right)\subset\mathbb{R}^7$ be a flat immersion defined as
$$\begin{aligned}
\psi_b(x,y)=\bigg(&\frac{b}{2\pi}\psi^*_{K,1},\frac{\sqrt{1-b^2}}{4\pi}\cos 4\pi x,
\frac{\sqrt{1-b^2}}{4\pi}\sin 4\pi x\bigg),
\end{aligned}$$
where 
\begin{equation}
    \label{eq-psi-KN}
    \begin{split}
        \psi^*_{K,Z}:=&\left(\frac{\sqrt{3}}{2}\cos\frac{2\pi Zy}{b},
\frac{1}{2}\cos\frac{2\pi Zy}{b}\cos 4\pi x,\frac{1}{2}\cos\frac{2\pi Zy}{b}\sin 4\pi x,\right.\\
&\qquad\left. \sin\frac{2\pi Zy}{b}\cos 2\pi x,\sin\frac{2\pi Zy}{b}\sin 2\pi x\right)\subset\mathbb{S}^4\subset\R^5,~~ Z\in\mathbb{Z}^+.    \end{split}
\end{equation}
A straightforward computation shows 
\begin{equation}\label{eq WEb}
\mathcal{W}(\psi_b)=\frac{\pi^2}{2}\left(\frac{1}{b}-\frac{3}{2}b^3+6b
\right).
\end{equation}
In particular, we have
$\mathcal{W}(\psi_b)<8\pi, $ $\hbox{ when }\,\,\,\, b_5<b<b_6,$
where $b_5=\frac{\sqrt{\mathfrak a}-\sqrt{8-\mathfrak a-\frac{64}{3\pi\sqrt{\mathfrak a}}}}{2}\approx 0.300$ and $b_6=\frac{\sqrt{\mathfrak a}+\sqrt{8-\mathfrak a-\frac{64}{3\pi\sqrt{\mathfrak a}}}}{2}\approx 0.688$. Here we set $\mathfrak a=\frac{8}{3}+\frac{4\sqrt{2}}{3}
\cos\left(\frac{1}{3}\arccos\left(\frac{\sqrt{2}(48-5\pi^2)}{\pi^2}\right)-\frac{4\pi}{3}\right)$. 

\end{example}
 
We can compare $\W(\psi_b)$ and the lower bound in Theorem \ref{thm-iso} in the following figure.
\begin{figure}[H]
    \centering
    \includegraphics[width=0.7\textwidth]{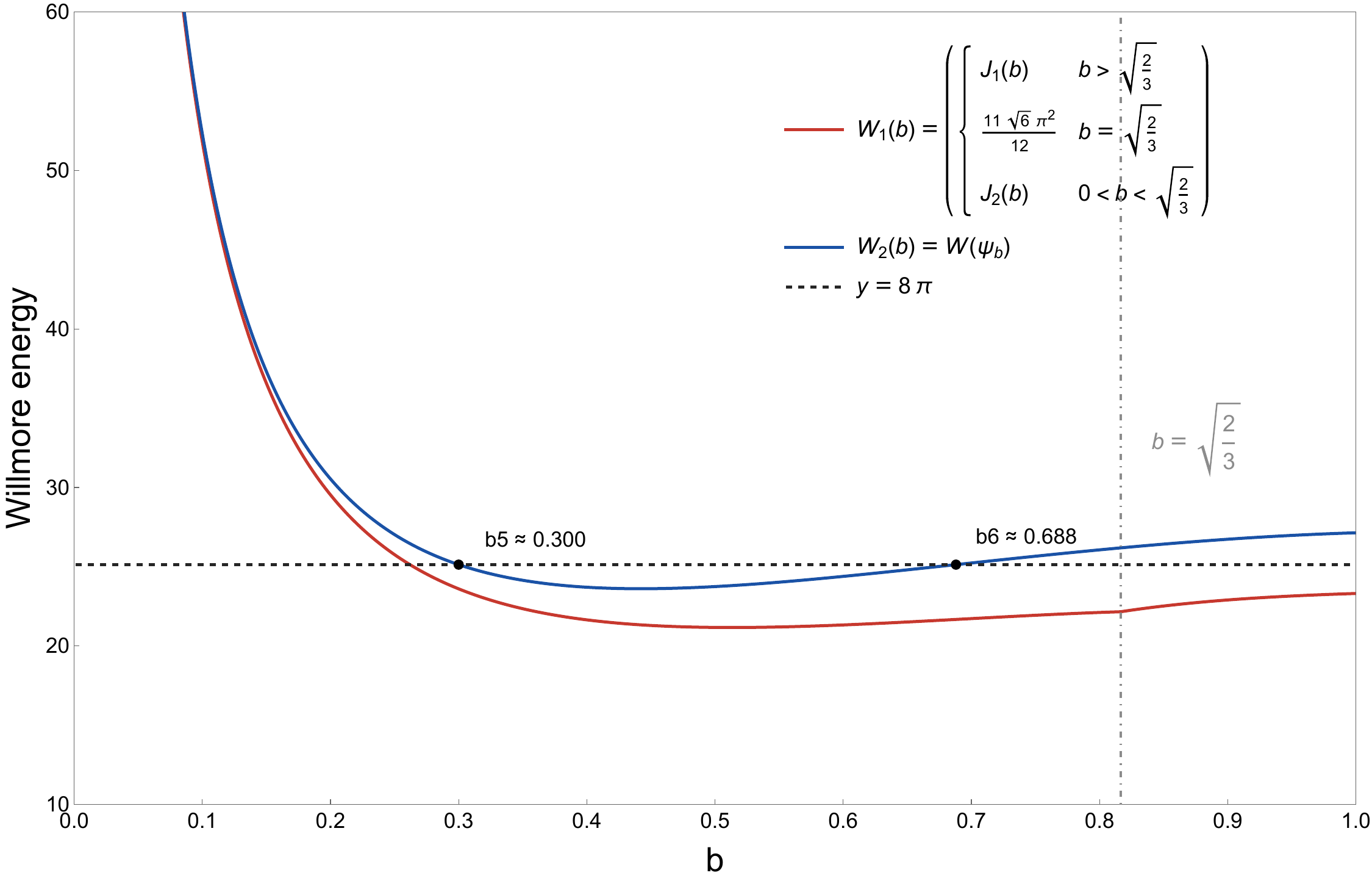} 
    \caption{The lower bound in Theorem~\ref{thm-iso}
and $\W(\psi_b)$ for $0<b\le1$.}
    \label{fig:willmore}
\end{figure}
    \begin{enumerate}
        \item When $b=\sqrt{\frac{2-\sqrt{2}}{3}}$, we obtain the minimum of $\W(\psi_b)\approx7.516\pi>6.682\pi$.
        \item When $b=\frac{2}{\pi}\mathrm{E}\left(\frac{\sqrt{3}}{2}\right)\approx0.771,$ we get $\W(\psi_b)\approx 8.224\pi<\W(\psi_T)\approx10.95\pi$,  Tompkins' Klein bottle, by \eqref{eq-Tompkins}. 
\item   When $b=1$, we obtain 
    $
    (\psi_b)|_{b=1}\subset\mathbb{S}^4(\frac{1}{2\pi})\subset\mathbb{R}^5$ with $\W(\psi_{b=1})=\frac{11\pi^2}{4}\approx8.639\pi$.
    \end{enumerate}



For $b>1$, a natural way to construct isometric immersions of flat Klein bottles with relatively small Willmore energy is to combine two neighboring vertical frequencies bracketing $b$. Setting
$ \mathfrak m=\lfloor b\rfloor,$ $Z=\mathfrak m+1$, we obtain the following example.
\begin{example}\label{exmaple b>1}
Let $b>1$ and let $\psi_{b,Z}(x,y):\mathbb{K}^2_b\to \mathbb{S}^9(\frac{1}{2\pi})\subset\mathbb R^{10}$  be the isometric immersion 
$$
\begin{aligned}
\psi_{b,Z}(x,y)=\bigg(\frac{B_1}{2\pi}\psi^*_{K,\mathfrak m}, \frac{B_2}{2\pi}\psi^*_{K,Z}
\bigg),~ \hbox{ with } B_1=\frac{\sqrt{Z^2-b^2}}{\sqrt{Z^2-\mathfrak m^2}}  \hbox{ and }  B_2=\frac{\sqrt{b^2-\mathfrak m^2}}{\sqrt{Z^2-\mathfrak m^2}}.
\end{aligned}
$$
\end{example}
\begin{remark}
Since the immersion in Example \ref{exmaple b<1} satisfies the balance condition and is isometric, we may rescale it to obtain an immersion into $\mathbb S^6$:
\begin{equation}\label{eq-A-b-1}
    \widetilde{\psi}_{b,1}(x,y)
=
\frac{1}{\sqrt{1+3b^2}}
\left(
2b\psi_{K,1}^{*}(x,y),\,
\sqrt{1-b^2}\cos4\pi x,\,
\sqrt{1-b^2}\sin4\pi x
\right)
\in \mathbb S^6.
\end{equation}
Its area is $
A(\widetilde{\psi}_{b,1})
=
\frac{8\pi^2b}{1+3b^2}$. When $b\to0$, we have $
A(\widetilde{\psi}_{b,1})\to0$. This agrees with the behavior of the lower bound in Theorem \ref{thm:improved-lower-bound} when $b\to0$. The corresponding comparison is shown in the following figure.
\begin{figure}[H]
    \centering
    \includegraphics[width=0.7\textwidth]{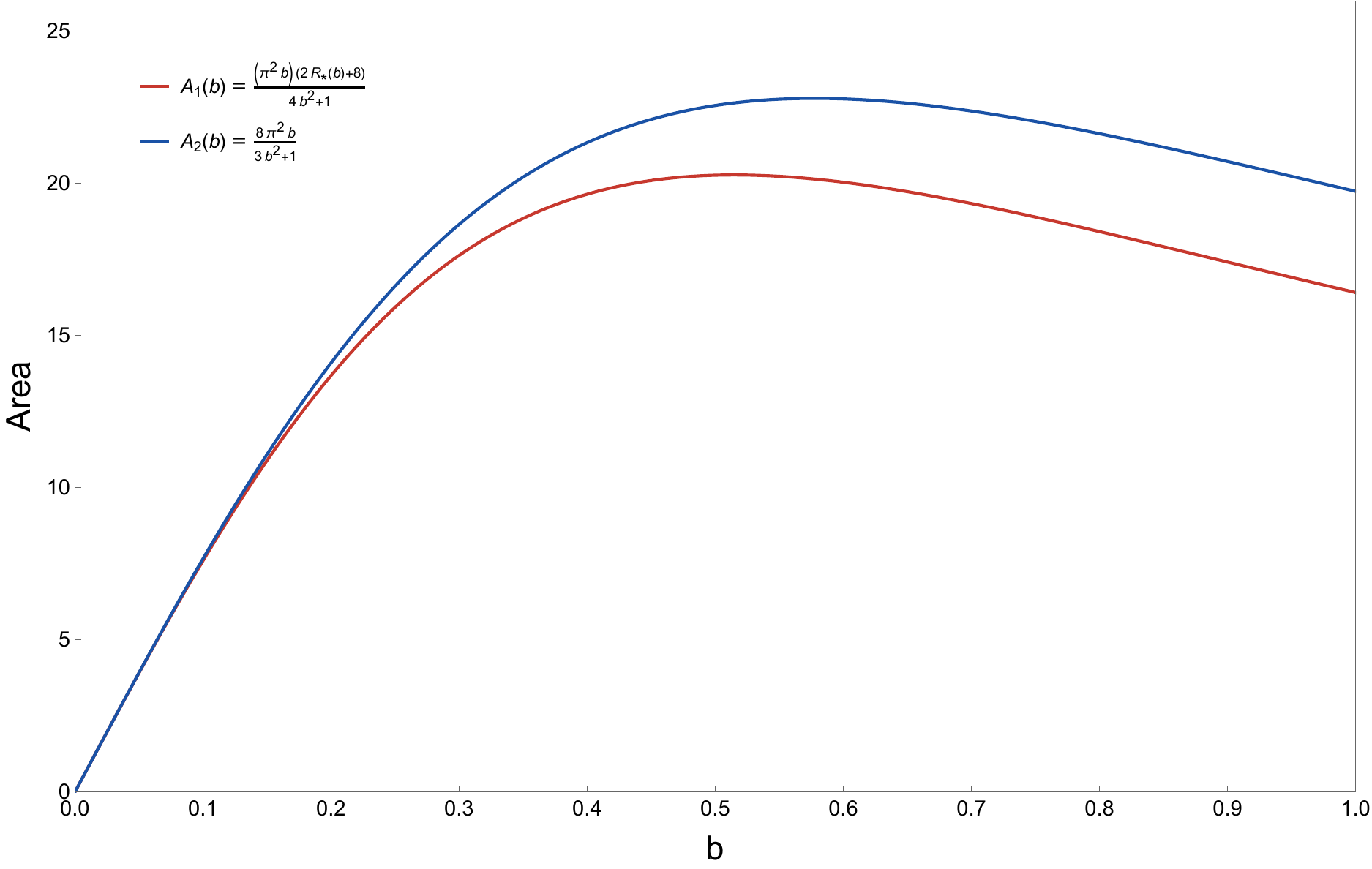} 
    \caption{Comparison of balanced area for $0<b\le 1$.}
    \label{fig:Area-1}
\end{figure}
As seen from the figure above, when $b\to0$, the area of the explicit immersion in (\ref{eq-A-b-1}) has the same asymptotic behavior as the lower bound in Theorem \ref{thm:improved-lower-bound}. Thus, the estimate in Theorem \ref{thm:improved-lower-bound} is consistent with the explicit construction in this regime. When $b>1$, after a similar argument, we obtain
$A(\widetilde\psi_{b,Z})=2\pi^2 b.$
 When $b\to\infty$, the area of the above example much fast than  the bound in Theorem \ref{thm:improved-lower-bound}. 
\end{remark}
 \subsection{Deformation of Tompkins' Klein bottle} Since the Klein bottles in Example \ref{exmaple b>1} have large Willmore energy when $b>1$. We turn to deform Tompkins' Klein bottle \cite{T-1941} in
$\mathbb{R}^5$ to get examples with better behavior.
\begin{example}
    For $a\geq 0$, consider
    \[
\psi_a(x,y)
=
\left(
a\cos y,\,
\cos y\cos x,\,
\cos y\sin x,\,
2\sin y\cos\frac{x}{2},\,
2\sin y\sin\frac{x}{2}
\right).
\]
In particular, $a=0$ gives Tompkins' Klein bottle.  
A direct computation gives
$|(\psi_a)_x|^2=1,$
$
|(\psi_a)_y|^2=4+(a^2-3)\sin^2y
$ and
$\langle(\psi_a)_x,(\psi_a)_y\rangle=0.$
Then
\[
g_a
=
|\mathrm d\psi_a|^2 
=
\mathrm dx^2+\mathrm dw^2, \hbox{ where } w(y)
=
\int_0^y
\sqrt{4+(a^2-3)\sin^2u}\,\mathrm du.
\]
Hence $\psi_a(x,w)$ is isometric in the coordinates $(x,w)$. To determine the parameter $b(a)$ of $\psi_a(x,w)$, set
$
Y(a)
:=
\int_0^{2\pi}
\sqrt{4+(a^2-3)\sin^2u}\,\mathrm du.
$
It is easy to see that $\psi_a(x,w)$ is invariant under the following three transformations:
\[
\begin{split}
\tau_1:(x,w)&\longmapsto(x+4\pi,w),\qquad
\tau_2:(x,w) \longmapsto(x,w+Y(a)),\qquad
\tau_3:(x,w)\longmapsto(x+2\pi,-w).
\end{split}
\]
By (\ref{eq-tau}), we get
\[
b(a)
=
\frac{Y(a)}{4\pi}
=
\frac{1}{\pi}
\int_0^{\frac{\pi}{2}}
\sqrt{4+(a^2-3)\sin^2u}\,\mathrm du
\in
\left[
\frac{2}{\pi}E\left(\frac{\sqrt3}{2}\right),
+\infty
\right).
\]
Set 
$L_a=\sqrt{4+(a^2-3)\sin^2y}$. A direct computation shows the Willmore energy  $\psi_a$ is
\[
\mathcal W(\psi_a)
=
\int
\left|
\frac12\Delta_{g_a}(\psi_a)
\right|^2
\,\mathrm dA_{g_a}
=
2\pi
\int_0^{\frac{\pi}{2}}
\left[
\left(1-\frac34\sin^2y\right)L_a
+\frac{2}{L_a}
+\frac{4(a^2+1)}{L_a^5}
\right]
\,\mathrm dy.
\]
 
\begin{figure}[H]
    \centering
    \includegraphics[width=0.7\textwidth]{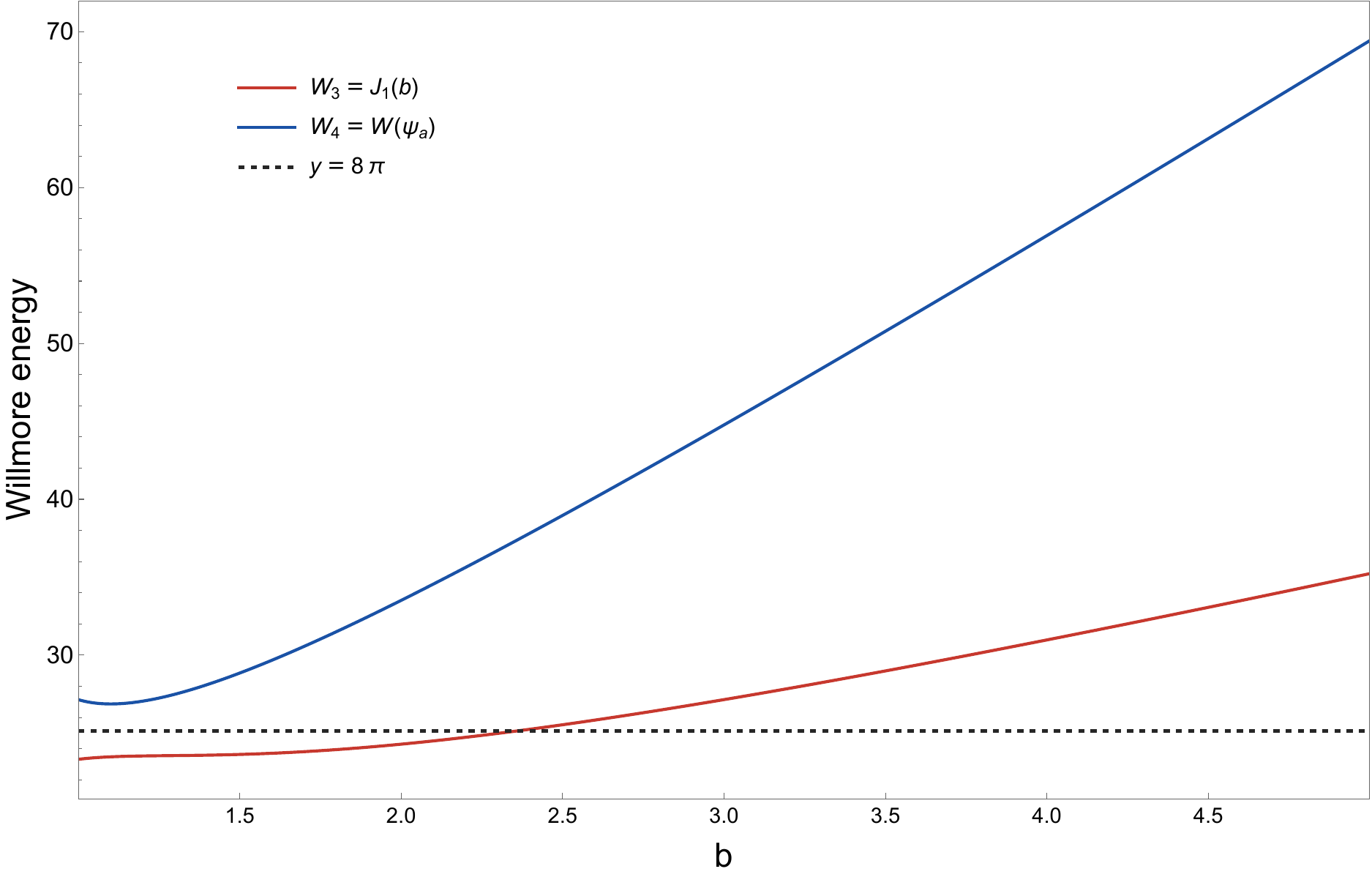} 
    \caption{The lower bound in Theorem~\ref{thm-iso}
vs. $\W(\psi_a)$ for $1<b\le5$.}
    \label{fig:willmore-0.8}
\end{figure}
\end{example}
\subsection{Construction using even and odd functions}
Motivated by the symmetries of the Klein bottle, we consider
\[
\widetilde{\psi}_b(x,y)
=
\frac{1}{2\pi}
\bigl(
U_b(y),
V_b(y)\cos 4\pi x,
V_b(y)\sin 4\pi x,
W_b(y)\cos 2\pi x,
W_b(y)\sin 2\pi x
\bigr).
\]
To satisfy the deck transformation $\tau_b$, we require
$U_b(y)$ and $V_b(y)$ to be even functions and $W_b(y)$ to be an
odd function. We choose the simplest odd function, that is,
$W_b(y):=C\sin\frac{2\pi}{b}y,
$ where $C$ is a constant.  For the functions $U_b$, $V_b$, and $W_b$, the isometry conditions then give
\[
1=\bigl|(\widetilde{\psi}_b)_x\bigr|^2=4V_b^2+W_b^2
,\qquad
1=\bigl|(\widetilde{\psi}_b)_y\bigr|^2
=
\frac{1}{4\pi^2}
\left[
(U_b')^2+(V_b')^2+(W_b')^2
\right].
\]
Solving these equations for $U_b$, $V_b$, and $W_b$ yields
the following example.

\begin{example}
For every $0<b<1$, the isometric immersion
$\widetilde{\psi}_b:\KB\to\mathbb{R}^5$ is defined by
\[
\widetilde{\psi}_b(x,y)
=
\frac{1}{2\pi}
\bigl(
U_b(y),
V_b(y)\cos 4\pi x,
V_b(y)\sin 4\pi x,
W_b(y)\cos 2\pi x,
W_b(y)\sin 2\pi x
\bigr),
\]
where $
U_b'(y)
=
\pi\sin\frac{2\pi}{b}y
\sqrt{
\frac{
4-b^2-3b^2\sin^2\frac{2\pi}{b}y
}{
1-b^2\sin^2\frac{2\pi}{b}y
}
}$ with 
$U_b(0)=0$,
$
V_b(y)
=
\frac12
\sqrt{1-b^2\sin^2\frac{2\pi}{b}y}$ and $W_b(y)=b\sin\frac{2\pi}{b}y$. 
A direct calculation gives 
\begin{equation}\label{eq-W-sharp}
    \begin{aligned}
\mathcal{W}(\widetilde{\psi}_b)
=
\frac18\int_0^1\int_0^b
\bigl|\Delta\widetilde{\psi}_b\bigr|^2\,dy\,dx
=
\frac{\pi^2}{4b^3}
\left[
-3b^6+8b^4+6b^2
+
\frac{
(4-b^2)^{\frac{3}{2}}+4b^4-b^2-8
}{
\sqrt{1-b^2}
}
\right].
\end{aligned}
\end{equation}
In particular,
$\mathcal{W}(\widetilde{\psi}_b)<8\pi$ when 
$b_7<b<b_8$ with $b_7\approx0.266$ and $b_8\approx0.906$.
\end{example}


\begin{figure}[H]
    \centering
    \includegraphics[width=0.7\textwidth]{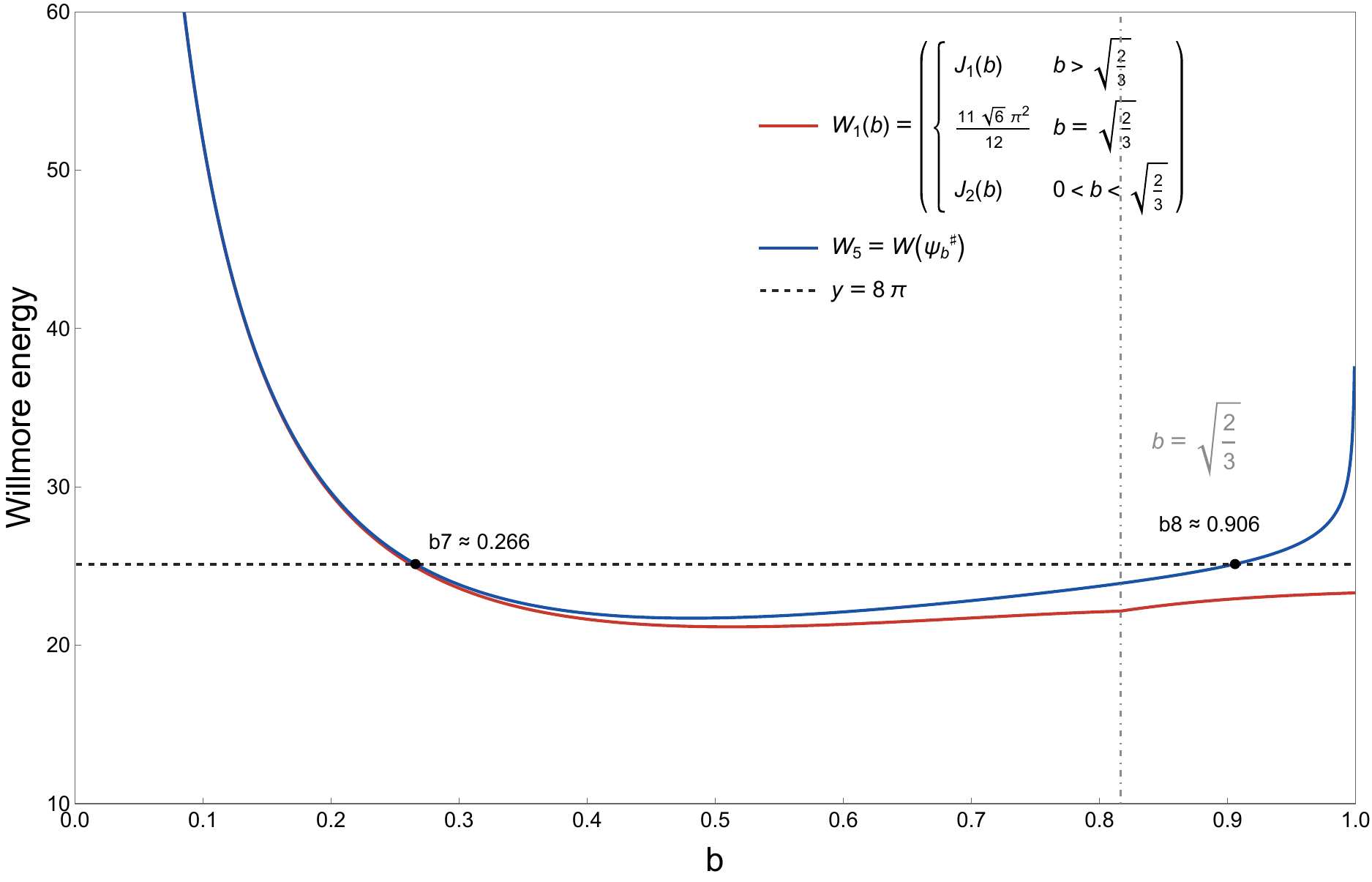} 
    \caption{The bound in Theorem~\ref{thm-iso}
vs. $\W(\widetilde{\psi}_b)$ for $0<b<1$.}
    \label{fig:willmore-0.8}
\end{figure}
Taking $b=\frac12$ in \eqref{eq-W-sharp}, we obtain $
\mathcal W(\widetilde{\psi}_{\frac{1}{2}})
=
\left(
\frac{125}{32}
+\frac{15\sqrt5}{2}
-\frac{32}{\sqrt3}
\right)\pi^2
\approx 2.201551\,\pi^2
<2.202\,\pi^2.$
Together with Theorem~\ref{thm H-B-M iso}, we obtain the following theorem.
\begin{theorem}
Let
\[
\beta_{\mathrm{flat}}
:=
\inf\left\{
\mathcal{W}(\psi)
\,\middle|\,
\psi:\KB\to\rn
\text{ is an isometric immersion}
\right\},
\]
where the infimum is taken over all $b>0$ and all ambient
dimensions $N$. Then
\[
2.202\pi^2
>
\beta_{\mathrm{flat}}
\geq
J_2(b_*)>2.136\,\pi^2,
\]
where $J_2(b)$ and $b_*$ are defined by \eqref{eq-J2} and \eqref{eq-b*},
respectively.
\end{theorem}
\begin{remark}
The function $\mathcal W(\widetilde{\psi}_b)$ has a unique global minimum
on $(0,1)$, numerically equal to approximately $2.200306\,\pi^2$.
Since this gives only a slight improvement over the choice $b=\frac{1}{2}$,
we omit the minimization argument.
\end{remark}

 \section{nonexistence of minimal isometric flat Klein bottle in $\sn$}\label{section8}

In the recent work of L\"{u}-Wang-Xie \cite[Corollary 2.8]{L-X-W}, they proved that the image of a $k$-dimensional homogeneous minimal torus in $\sn$ is a torus for any $k\geq2$. In particular, there exist no flat minimal Klein bottles in $\sn$. 
In this section, we give a different proof for the non-existence of flat minimal Klein bottle in $\sn$. 
    \begin{theorem}
        There exists no minimal isometric immersion from $(\KB,e^{2\rho}\dd s_0^2)$ into $\sn$ with $\rho\equiv const$ for any $b>0$.
    \end{theorem}
    \begin{proof}
     Suppose there exists $X:\KB \longrightarrow \mathbb{S}^{n}$ is a minimal flat immersion from $\left(\KB,e^{2\rho}\dd s^2_0\right)$ into $\sn$. Let $E_\lambda$ be the eigenspace of the Laplacian of $|\dd X|^2=e^{2\rho}\dd s_0^2$ corresponding to the eigenvalue $\lambda$. Then the coordinate functions $X_1\dots,X_{n+1}$ of $X$
are all contained in $E_{\lambda=2}$. On the other hand, by \eqref{eq:eigenfunctions} and \eqref{eq:fourier-expansion-map}, we can expand $X$ as follows:
$$X=\sum\limits_i(A_{p_iq_i}\bar{f}_{p_iq_i}+B_{p_iq_i}\bar{g}_{p_iq_i})+\sum\limits_j(A_{r_js_j}\bar{h}_{r_js_j}+B_{r_js_j}\bar{l}_{r_js_j}),$$
where $(p_i,q_i)\in S_1$, $(r_j,s_j)\in S_2$ and $A_{p_iq_i},B_{p_iq_i},A_{r_js_j},B_{r_js_j}\in \mathbb{R}^{n+1}$. Since $\rho=const$, we get
\begin{equation*}
    \Delta_\rho=e^{-2\rho}\Delta_0,
\end{equation*}
where $\Delta_{\rho}$ denotes the Laplace operator with respect to the metric $|\mathrm dX|^2$. Since $\{X_1,\cdots,X_{n+1}\}\subset E_{\lambda=2}$, it follows that 
\begin{equation}\label{eq eigenvalue}
    2e^{2\rho}=4\pi^2(p_i^2+\frac{q_i^2}{b^2})=4\pi^2(r_j^2+\frac{s_j^2}{b^2}),~~ \forall ~i,j.
\end{equation}
Moreover, since $p_i$ is even and $r_j$ is odd, we have $p_i\neq r_j$ for all $i$ and $j$. Hence, by \eqref{eq eigenvalue},
\begin{equation}\label{neq}
    q_i\neq s_j,~~\forall~ i,j.
\end{equation}
We also have \begin{align*}
    1=\int_0^1|X|^2\dd x&=\frac{1}{b}\sum_{p_i>0}a_{p_i0}+\frac{2}{b}\sum_{(p_i,q_i)\in S_1,q_i>0}\left(a_{p_iq_i}\cos^2\left(\frac{2\pi q_iy}{b}\right)\right)\\&\quad+\frac{2}{b}\sum_{(r_j,s_j)\in S_2}\left(a_{r_js_j}\sin^2\left(\frac{2\pi s_jy}{b}\right)\right)\\
    &=\frac{1}{b}\left(\sum_{p_i>0}a_{p_i0}+\sum_{(p_i,q_i)\in S_1,q_i>0}a_{p_iq_i}+\sum_{(r_j,s_j)\in S_2}a_{r_js_j}\right)\\
    &\quad+\frac{1}{b}\sum_{(p_i,q_i)\in S_1,q_i>0}\left(a_{p_iq_i}\cos\left(\frac{4\pi q_iy}{b}\right)\right) -\frac{1}{b}\sum_{(r_j,s_j)\in S_2}\left(a_{r_js_j}\cos\left(\frac{4\pi s_jy}{b}\right)\right),
\end{align*}
It follows from (\ref{neq}) that $a_{p_iq_i} = a_{r_js_j} = 0$ whenever $q_i\neq 0$ and $s_j\neq0$. Hence, we get  
$X=A_{p_i0} \bar{f}_{p_i0}+B_{p_i0}\bar{g}_{p_i0}.$ In particular, $|X_y|^2\equiv0$, contrary to the isometric condition of $X$.
    \end{proof}

\appendix
\section{}\label{section A}
In the Appendix, we provide proofs of the technical lemmas in Section \ref{section4} and \ref{section5}.
\subsection{Auxiliary Lemmas in Section \ref{section4}}
\begin{lemma}\label{lemma-id-u}
We retain the notion in Section \ref{section4}. The following identity holds:
    \begin{equation}
\frac {\mathfrak u}{2}
+(4b^2+3)\operatorname{Re}
\bigl(C_{2,0}\overline{C_{2,-2}}\bigr)
+\mathcal J=0,
\label{eq:lemma34-decomposition}
\end{equation}
where $
\mathcal J
:=\sum_{(m,n)\in\widetilde\Sigma}
Q_{mn}C_{mn}\overline{C_{m,n-2}}$ with $Q_{mn}=\frac{b^2m^2-(n-1)^2+4}{4}$ and $$\widetilde\Sigma
:=
\widehat\Sigma
\setminus
\left(
\{(0,3),(\pm2,2),(0,2)\}
\cup
\{(m,2):m\in2\mathbb Z+1\}
\right).$$
\end{lemma}
\begin{proof}
    Define
\[
F:=\frac34(|\psi|^2-1)
+\frac{b^2}{16\pi^2}(|\psi_x|^2-|\psi_y|^2),
\qquad
[F]_{mn}:=\frac1b\int_{\T_b^2}
F\mathrm e^{-2\pi \mathrm i\left(mx+\frac{ny}{b}\right)}\,\mathrm dM_0.
\]
Clearly $F\equiv0$, so $[F]_{02}=0$. Using the computations in Proposition \ref{prop:fourier-lower-bound}, we obtain
\begin{equation}\label{eq:lemma34-fourier-identity}
\begin{split}
0
&=\frac34[|\psi|^2]_{02}
+\frac{b^2}{16\pi^2}
\left([|\psi_x|^2]_{02}-[|\psi_y|^2]_{02}\right)\\
&=\frac34\sum_{(m,n)\in\mathbb Z^*}
C_{mn} C_{-m,2-n}
\\&\qquad+\frac{b^2}{16\pi^2}
\left(
\sum_{(m,n)\in\mathbb Z^*}
4\pi^2m^2C_{mn} C_{-m,2-n}
-\sum_{(m,n)\in\mathbb Z^*}
\frac{4\pi^2}{b^2}n(n-2)C_{mn} C_{-m,2-n}
\right)\\
&=\frac34\sum_{(m,n)\in\mathbb Z^*}
C_{mn}\overline{C_{m,n-2}}
+\frac14\sum_{(m,n)\in\mathbb Z^*}
\bigl(b^2m^2-n(n-2)\bigr)
C_{mn}\overline{C_{m,n-2}}\\
&=\sum_{(m,n)\in\mathbb Z^*}
\frac{b^2m^2-(n-1)^2+4}{4}
C_{mn}\overline{C_{m,n-2}}\\
&=\sum_{(m,n)\in\mathbb Z^*}
Q_{mn}C_{mn}\overline{C_{m,n-2}}.
\end{split}
\end{equation}

Since the estimate in Proposition \ref{lem:u-upper} requires Lemma \ref{lem:A1}, we separate the terms in \eqref{eq:lemma34-fourier-identity} for which $w_{mn}\leq0$ or $w_{m,n-2}\leq0$. By Proposition \ref{prop:fourier-lower-bound}, $w_{mn}=0$ if and only if $ (m,n)=(0,\pm1)\ \text{or}\ (\pm2,0)$.
While $w_{m,n-2}=0$ if and only if $(m,n)=(0,3),\ (0,1),\ \text{or}\ (\pm2,2)$. We now compute the corresponding terms in \eqref{eq:lemma34-fourier-identity}. 
\begin{itemize}
    \item For $(m,n)=(0,1)$, \eqref{eq:deck-relation} gives
$C_{0,-1}=(-1)^0C_{01}=C_{01}$, and hence
$
Q_{01}C_{01}\overline{C_{0,-1}}
=|C_{01}|^2=\frac {\mathfrak u}{2}.
$
\item For $(m,n)=(0,-1)$ or $(0,3)$, we get
$Q_{0,-1}=Q_{03}=0.$
\item For $(m,n)=(\pm2,0)$ or $(\pm2,2)$, we obtain
$
Q_{\pm2,0}=Q_{\pm2,2}=\frac{4b^2+3}{4}.
$
\end{itemize}

On the other hand, $w_{mn}<0$ if and only if $(m,n)=(0,0)$ or $(\pm1,0)$. In either case,
$C_{00}=C_{\pm1,0}=0$.
Similarly, $w_{m,n-2}<0$ if and only if  $(m,n)=(0,2)$ or $(\pm1,2)$, and therefore
$\overline{C_{00}}=\overline{C_{\pm1,0}}=0$.

Moreover, if $(m,n)=(2j+1,2)$ for some $j\in\mathbb Z$, then the symmetry relation for the Klein bottle gives
$\overline{C_{2j+1,0}}=0$.
Thus, these terms do not contribute to \eqref{eq:lemma34-fourier-identity}. 
Hence, \eqref{eq:lemma34-fourier-identity} can be rewritten as (\ref{eq:lemma34-decomposition}).
\end{proof}

\begin{lemma}\label{lem:A1}
For every $(m,n)\in\widetilde\Sigma$, we have
\begin{equation}
|Q_{mn}|
\le\frac{4+b^2}{4}\sqrt{w_{mn}w_{m,n-2}}.
\label{eq:A1}
\end{equation}
Then it follows that
\begin{equation}
|\mathcal J|\le\frac{4+b^2}{4}R.
\label{eq:A2}
\end{equation}
\end{lemma}

\begin{proof}
For $(m,n)\in\widetilde\Sigma$, we have
$w_{mn}>0$ and $w_{m,n-2}>0$. By the definitions of $Q_{mn}$ and
$w_{mn}$, it is enough to prove that
$
\bigl(b^2m^2-(n-1)^2+4\bigr)^2
\le
(4+b^2)^2
\bigl[m^2+4(n^2-1)\bigr]
\bigl[m^2+4((n-2)^2-1)\bigr].
$
Set
\[
\begin{split}
G(m,n,b)
:={}&
(4+b^2)^2
\bigl[m^2+4(n^2-1)\bigr]
\bigl[m^2+4((n-2)^2-1)\bigr]-\bigl(b^2m^2-(n-1)^2+4\bigr)^2.
\end{split}
\]
It remains to show that $G(m,n,b)\ge0$ for every $b>0$ and every
$(m,n)\in\widetilde\Sigma$. We consider the following three cases.
\begin{enumerate}
    \item $n=1$:
Since $m^2\ge1$, it follows that
\[
\begin{split}
G(m,1,b)
&=(4+b^2)^2m^4-(b^2m^2+4)^2=8(m^2-1)(b^2m^2+2m^2+2)\ge0.
\end{split}
\]
\item
$n=0$ or $n=2$: In this case $m^2-16\ge0$, so we have
\[
\begin{split}
G(m,0,b)=G(m,2,b)
&=(4+b^2)^2(m^2-4)(m^2+12)
-(b^2m^2+3)^2\\
&>(4+b^2)^2(m^2+2)^2
-(b^2m^2+3)^2>(b^2m^2+8)^2
-(b^2m^2+3)^2\\
&>0.
\end{split}
\]
\item
$(n-1)^2\ge4$: We have $G(m,n,b)=G_4(n,b)m^4+G_2(n,b)m^2+G_0(n,b)$ with
$G_4(n,b)=8b^2+16>0$, $G_2(n,b)=8b^4(n-1)^2
+(66(n-1)^2-8)b^2
+128(n-1)^2\geq0$, and 
$G_0(n,b)=16b^2(b^2+8)(n-1)^2\bigl((n-1)^2-4\bigr)+\bigl((n-1)^2-4\bigr)
\bigl(255(n-1)^2+4\bigr)
\ge0$.
\end{enumerate}

Thus $G(m,n,b)\ge0$ for every $(m,n)\in\widetilde\Sigma$, and
\eqref{eq:A1} follows.

Moreover, using \eqref{eq:A1}, we obtain
\[
\begin{split}
|\mathcal J|
&\le
\sum_{(m,n)\in\widetilde\Sigma}
|Q_{mn}|\,
|C_{mn}\overline{C_{m,n-2}}|\le
\frac{4+b^2}{4}
\sum_{(m,n)\in\widetilde\Sigma}
\sqrt{w_{mn}}\,|C_{mn}|
\sqrt{w_{m,n-2}}\,|C_{m,n-2}|.
\end{split}
\]
By the Cauchy--Schwarz inequality, we get
\[
\begin{split}
|\mathcal J|
&\le
\frac{4+b^2}{4}
\left(
\sum_{(m,n)\in\widetilde\Sigma}
w_{mn}|C_{mn}|^2
\right)^{1/2}
\left(
\sum_{(m,n)\in\widetilde\Sigma}
w_{m,n-2}|C_{m,n-2}|^2
\right)^{1/2}\le\frac{4+b^2}{4}R.
\end{split}
\]
This proves \eqref{eq:A2}.
\end{proof}
\subsection{Proof of Lemma \ref{lem-piecewise-lower-bound}}\label{appendix-proof}
\begin{proof}[Proof of Lemma \ref{lem-piecewise-lower-bound}]
   \textbf{Case 1.} When $b>\sqrt{\frac{2}{3}}$, we estimate $D$ from below.   
   Set
$
\widetilde S_1
=
S_1\setminus\{(p,q)\in S_1\mid p=0\}$. For $(p,q)\in\widetilde S_1$, we have
\begin{equation*}
\begin{split}
\overline f_{p0}
&=
\sqrt{\frac{2}{b}}\cos(2\pi px),
\qquad
\overline g_{p0}
=
\sqrt{\frac{2}{b}}\sin(2\pi px),\\
\overline f_{pq}
&=
\frac{2}{\sqrt b}
\cos(2\pi px)
\cos\left(\frac{2\pi qy}{b}\right),
\qquad
\overline g_{pq}
=
\frac{2}{\sqrt b}
\sin(2\pi px)
\cos\left(\frac{2\pi qy}{b}\right).
\end{split}
\end{equation*}
Set
\begin{equation*}
v_q
=
\begin{cases}
\displaystyle \sqrt{\frac{2}{b}}, & q=0,\\[4pt]
\displaystyle \frac{2}{\sqrt b}, & q\ge1.
\end{cases} 
\end{equation*}
Differentiating $\psi$ with respect to $x$ and setting $y=0$, we obtain
\begin{equation}
\psi_x(x,0)
=
2\pi
\sum_{(p,q)\in S_1}
pv_q
\left[
-A_{pq}\sin(2\pi px)
+
B_{pq}\cos(2\pi px)
\right].
\label{eq:psix-y0}
\end{equation}

Since $\psi$ is isometric, it follows from \eqref{eq:psix-y0} that
$\widetilde S_1\cap\mathcal A^*(\psi)\neq\varnothing$.
For $p\in2\mathbb Z^+$, define
\begin{equation}
\begin{split}
A_p^e
=
\sum_{q\in2\mathbb N}v_qA_{pq},\qquad
A_p^o
=
\sum_{q\in2\mathbb N+1}v_qA_{pq},\qquad
B_p^e
=
\sum_{q\in2\mathbb N}v_qB_{pq},\qquad
B_p^o
=
\sum_{q\in2\mathbb N+1}v_qB_{pq}.
\end{split}
\label{eq:Ape-definition}
\end{equation}
Substituting \eqref{eq:Ape-definition} into \eqref{eq:psix-y0} gives
\begin{equation*}
\psi_x(x,0)
=
2\pi
\sum_{p\in2\mathbb Z^+}
p
\left[
-(A_p^e+A_p^o)\sin(2\pi px)
+
(B_p^e+B_p^o)\cos(2\pi px)
\right].
\end{equation*}

On the other hand, differentiating $\psi$ with respect to $x$ and setting $y=b/2$, we obtain
\begin{equation*}
\begin{split}
\psi_x\left(x,\frac b2\right)
&=
2\pi
\sum_{(p,q)\in S_1}
pv_q
\left[
-A_{pq}\sin(2\pi px)\cos(\pi q)
+
B_{pq}\cos(2\pi px)\cos(\pi q)
\right]\\
&=
2\pi
\sum_{(p,q)\in S_1}
pv_q
\left[
-(-1)^qA_{pq}\sin(2\pi px)
+
(-1)^qB_{pq}\cos(2\pi px)
\right].
\end{split}
\end{equation*}
Together with \eqref{eq:Ape-definition}, this yields
\begin{equation*}
\psi_x\left(x,\frac b2\right)
=
2\pi
\sum_{p\in2\mathbb Z^+}
p
\left[
-(A_p^e-A_p^o)\sin(2\pi px)
+
(B_p^e-B_p^o)\cos(2\pi px)
\right].
\end{equation*}
Since $\psi$ is isometric,
\begin{equation}
1
=
\int_0^1|\psi_x(x,0)|^2\,dx
=
2\pi^2
\sum_{p\in2\mathbb Z^+}
p^2
\left(
|A_p^e+A_p^o|^2
+
|B_p^e+B_p^o|^2
\right),
\label{eq:isometric-x-0}
\end{equation}
and
\begin{equation}
1
=
\int_0^1
\left|
\psi_x\left(x,\frac b2\right)
\right|^2\,dx
=
2\pi^2
\sum_{p\in2\mathbb Z^+}
p^2
\left(
|A_p^e-A_p^o|^2
+
|B_p^e-B_p^o|^2
\right).
\label{eq:isometric-x-b2}
\end{equation}
Adding \eqref{eq:isometric-x-0} and \eqref{eq:isometric-x-b2}, we obtain
(\ref{eq:sum-Ap}).

If $(p,q)\in\widetilde S_1$, then $p^2\ge4$ and $q^4-q^2\ge0$. Hence, by \eqref{eq:d-positive},
\begin{equation*}
\frac{d_{pq}}{p^2}
=
p^2-1-\frac{2}{b^2}
+\frac{2q^2}{b^2}
+
\frac{q^4-q^2}{b^4p^2}
\ge
3-\frac{2}{b^2}+\frac{2q^2}{b^2}.
\end{equation*}
Therefore,
\begin{equation}
d_{pq}
\ge
p^2
\left(
3-\frac{2}{b^2}+\frac{2q^2}{b^2}
\right),
\qquad
\forall (p,q)\in\widetilde S_1.
\label{eq:d-pq-lower}
\end{equation}
Set $
\alpha_q
=
3-\frac{2}{b^2}+\frac{2q^2}{b^2}
$.Since $b>\sqrt{\frac{2}{3}}$, we have $\alpha_q>0$.

Now, we first estimate $|A_p^e|^2$ and $|B_p^e|^2$. By the Cauchy-Schwarz inequality,
\begin{equation}
|A_p^e|^2
=
\left|
\sum_{q\in2\mathbb N}v_qA_{pq}
\right|^2
\le
\left(
\sum_{q\in2\mathbb N}\frac{v_q^2}{\alpha_q}
\right)
\left(
\sum_{q\in2\mathbb N}\alpha_q|A_{pq}|^2
\right).
\label{eq:CS-even-A}
\end{equation}
Multiplying both sides by $p^2$ and using \eqref{eq:d-pq-lower}, we obtain
\begin{equation}
p^2|A_p^e|^2
\le
C^e(b)
\sum_{q\in2\mathbb N}
d_{pq}|A_{pq}|^2,
\label{eq:even-A}
\end{equation}
where $
C^e(b)
=
\sum_{q\in2\mathbb N}\frac{v_q^2}{\alpha_q}$. Similarly,
\begin{equation}
p^2|B_p^e|^2
\le
C^e(b)
\sum_{q\in2\mathbb N}
d_{pq}|B_{pq}|^2.
\label{eq:even-B}
\end{equation}
Using the Mittag-Leffler expansion, we obtain
\begin{equation}
C^e(b)
=
\frac{2}{b}
\sum_{q\in2\mathbb Z}
\frac{1}{
3-\frac{2}{b^2}+\frac{2q^2}{b^2}
}
=
\frac{\pi b}{
2\DD
}
\coth\left(
\frac{\pi\DD}{2}
\right),
\label{eq:Ce}
\end{equation}
with $\DD:=\sqrt{\frac{3b^2-2}{2}}$.

We next estimate $|A_p^o|^2$ and $|B_p^o|^2$. Again by the Cauchy-Schwarz inequality,
\begin{equation}
\begin{split}
p^2|A_p^o|^2
&\le
p^2
\left(
\sum_{q\in2\mathbb N+1}
\frac{v_q^2}{\alpha_q}
\right)
\left(
\sum_{q\in2\mathbb N+1}
\alpha_q|A_{pq}|^2
\right)\le
C^o(b)
\sum_{q\in2\mathbb N+1}
d_{pq}|A_{pq}|^2,
\end{split}
\label{eq:odd-A}
\end{equation}
where $C^o(b)
=
\sum_{q\in2\mathbb N+1}
\frac{v_q^2}{\alpha_q}
$. Similarly,
\begin{equation}
p^2|B_p^o|^2
\le
C^o(b)
\sum_{q\in2\mathbb N+1}
d_{pq}|B_{pq}|^2.
\label{eq:odd-B}
\end{equation}
Combining this with \eqref{eq:Ce}, a direct computation gives
\begin{equation*}
\begin{split}
C^o(b)
&=
\frac{2}{b}
\sum_{q\in\mathbb Z}
\frac{1}{
3-\frac{2}{b^2}+\frac{2q^2}{b^2}
}
-
\frac{2}{b}
\sum_{q\in2\mathbb Z}
\frac{1}{
3-\frac{2}{b^2}+\frac{2q^2}{b^2}
}\\
&=
b
\sum_{q\in\mathbb Z}
\frac{1}{
\DD^2+q^2
}
-
\frac{\pi b}{
2\DD
}
\coth\left(
\frac{\pi\DD}{2}
\right)\\
&=
\frac{\pi b}{
\DD
}
\coth\left(
\pi\DD
\right)
-
\frac{\pi b}{
2\DD
}
\coth\left(
\frac{\pi\DD}{2}
\right)\\
&=
\frac{\pi b}{
2\DD
}
\tanh\left(
\frac{\pi\DD}{2}
\right),
\end{split}
\end{equation*}
Since $\DD>0$, it follows that
\begin{equation}
\begin{split}
C^e(b)-C^o(b)
&=
\frac{\pi b}{
2\DD
}
\left[
\coth\left(
\frac{\pi\DD}{2}
\right)
-
\tanh\left(
\frac{\pi\DD}{2}
\right)
\right]>0.
\end{split}
\label{eq:Ce-Co}
\end{equation}

Since $\widetilde S_1\cap\mathcal A^*(\psi)\neq\varnothing$, there exist $p\in2\mathbb Z^+$ and $q\in\mathbb{N}$ such that $a_{pq}\neq0$. Combining
\eqref{eq:sum-Ap}, \eqref{eq:even-A}, \eqref{eq:even-B},
\eqref{eq:odd-A}, \eqref{eq:odd-B}, and \eqref{eq:Ce-Co}, we obtain
\begin{equation}
\begin{split}
D
&\ge
\sum_{(p,q)\in S_1^*}d_{pq}a_{pq}\ge
\frac{1}{C^e(b)}
\sum_{p\in2\mathbb Z^+}
p^2
\left(
|A_p^e|^2+|B_p^e|^2
\right)
+
\frac{1}{C^o(b)}
\sum_{p\in2\mathbb Z^+}
p^2
\left(
|A_p^o|^2+|B_p^o|^2
\right)\\
&\ge
\frac{1}{C^e(b)}
\sum_{p\in2\mathbb Z^+}
p^2
\left(
|A_p^e|^2
+
|A_p^o|^2
+
|B_p^e|^2
+
|B_p^o|^2
\right)\\
&=
\frac{1}{C^e(b)}
\frac{1}{2\pi^2}.
\end{split}
\label{eq:D-lower}
\end{equation}
Substituting \eqref{eq:D-lower} into (\ref{eq-W-est}) yields
\begin{equation}\label{ineq-b-big}
\mathcal W(\psi_K)
\ge
\frac{\pi^2}{2}
\left(
b+\frac{3}{b}
\right)
+
\frac{2\pi\DD}{b}
\tanh\left(
\frac{\pi\DD}{2}
\right).
\end{equation}

It remains to show that equality cannot occur. Assume that equality holds in \eqref{eq:D-lower}, then equality must hold throughout
\eqref{eq:d-pq-lower}, \eqref{eq:CS-even-A},
\eqref{eq:even-A}, \eqref{eq:even-B},
\eqref{eq:odd-A}, and \eqref{eq:odd-B}.
Firstly, from \eqref{eq:d-pq-lower},
\begin{equation*}
d_{pq}-p^2\alpha_q
=
p^2(p^2-4)
+
\frac{q^2(q^2-1)}{b^4}
\ge0,
\qquad
\forall p\in2\mathbb Z^+,\ q\in2\mathbb Z,
\end{equation*}
with equality if and only if $(p,q)=(2,0)$.

On the other hand, equality in \eqref{eq:CS-even-A} means that, for each fixed $p\in2\mathbb Z^+$, there exists $V_p\in\rn$ such that
\begin{equation*}
A_{pq}
=
\frac{v_q}{\alpha_q}V_p,
\qquad
\forall q\in2\mathbb N.
\end{equation*}
If there exists $p\in2\mathbb Z^+$ such that $V_p\neq0$, then
\begin{equation*}
\sum_{q\in2\mathbb N}
\left(
d_{pq}-p^2\alpha_q
\right)
|A_{pq}|^2
>0,
\end{equation*}
which contradicts the equality case in \eqref{eq:even-A}. Hence
$V_p=0$ for every $p\in2\mathbb Z^+$. Applying the same argument to $B_{pq}$ gives
\begin{equation*}
|A_p^e|^2
=
|B_p^e|^2
=
0,
\qquad
\forall p\in2\mathbb Z^+.
\end{equation*}
Moreover, equality in \eqref{eq:D-lower}, together with
\eqref{eq:Ce-Co}, similarly implies
\begin{equation*}
|A_p^o|^2
=
|B_p^o|^2
=
0,
\qquad
\forall p\in2\mathbb Z^+.
\end{equation*}
This contradicts \eqref{eq:sum-Ap}. Hence equality cannot hold in (\ref{ineq-b-big}).

\textbf{Case 2.} When $0<b<\sqrt{\frac{2}{3}}$, we estimate $\widetilde D$ from below.

    For $s\in\mathbb N_+$, we have
\begin{equation*}
\overline h_{rs}
=
\frac{2}{\sqrt b}h_{rs},
\qquad
\overline l_{rs}
=
\frac{2}{\sqrt b}l_{rs}.
\end{equation*}
Differentiating $\psi$ with respect to $y$ and evaluating at $y=0$ and $y=b/2$, respectively, gives
\begin{equation*}
\psi_y(x,0)
=
\frac{4\pi}{b\sqrt b}
\sum_{(r,s)\in S_2}
s
\left[
A_{rs}\cos(2\pi rx)
+
B_{rs}\sin(2\pi rx)
\right],
\end{equation*}
\begin{equation*}
\psi_y\left(x,\frac b2\right)
=
\frac{4\pi}{b\sqrt b}
\sum_{(r,s)\in S_2}
s
\left[
(-1)^sA_{rs}\cos(2\pi rx)
+
(-1)^sB_{rs}\sin(2\pi rx)
\right].
\end{equation*}
For each fixed $r\in2\mathbb Z+1$, define
\begin{equation*}
A_r^e
=
\sum_{s\in2\mathbb Z^+}sA_{rs},\qquad
B_r^e
=
\sum_{s\in2\mathbb Z^+}sB_{rs},\qquad
A_r^o
=
\sum_{s\in2\mathbb N+1}sA_{rs},\qquad
B_r^o
=
\sum_{s\in2\mathbb N+1}sB_{rs}.
\end{equation*}
Thus
\begin{equation*}
\psi_y(x,0)
=
\frac{4\pi}{b\sqrt b}
\sum_{r\in2\mathbb N+1}
\left[
(A_r^e+A_r^o)\cos(2\pi rx)
+
(B_r^e+B_r^o)\sin(2\pi rx)
\right],
\end{equation*}
\begin{equation*}
\psi_y\left(x,\frac b2\right)
=
\frac{4\pi}{b\sqrt b}
\sum_{r\in2\mathbb N+1}
\left[
(A_r^e-A_r^o)\cos(2\pi rx)
+
(B_r^e-B_r^o)\sin(2\pi rx)
\right].
\end{equation*}
Since $\psi$ is isometric, we get
\begin{equation}
\begin{split}
2
&=
\int_0^1
\left(
|\psi_y(x,0)|^2
+
\left|
\psi_y\left(x,\frac b2\right)
\right|^2
\right)\,\mathrm dx\\
&=
\frac{8\pi^2}{b^3}
\sum_{r\in2\mathbb N+1}
\left(
|A_r^e+A_r^o|^2
+
|B_r^e+B_r^o|^2
\right)
+
\frac{8\pi^2}{b^3}
\sum_{r\in2\mathbb N+1}
\left(
|A_r^e-A_r^o|^2
+
|B_r^e-B_r^o|^2
\right)\\
&=
\frac{16\pi^2}{b^3}
\sum_{r\in2\mathbb N+1}
\left(
|A_r^e|^2
+
|A_r^o|^2
+
|B_r^e|^2
+
|B_r^o|^2
\right).
\end{split}
\label{eq:sum-Ar}
\end{equation}

We now observe that
\begin{equation*}
\begin{split}
\widetilde d_{rs}
-
\frac{s^2(s^2-1+2b^2-3b^4)}{b^4}
&=
\left(
r^2+\frac{s^2}{b^2}
\right)^2
-
4r^2
-
\frac{s^2}{b^4}
-
\frac{s^2(s^2-1+2b^2-3b^4)}{b^4}\\
&=
r^4
+
\frac{2r^2s^2}{b^2}
+
\frac{s^4-s^2}{b^4}
-
4r^2
-
\frac{s^2(s^2-1+2b^2-3b^4)}{b^4}\\
&=
\frac{
s^4-s^2
+
r^2b^4(r^2-4)
+
2r^2s^2b^2
-
s^4+s^2
-
2b^2s^2
+
3b^4s^2
}{b^4}\\
&=
\frac{
b^4r^2(r^2-4)
+
2b^2s^2(r^2-1)
+
3b^4s^2
}{b^4}\\
&\ge
3(s^2-1)
\ge0.
\end{split}
\end{equation*}
Equality holds if and only if $(r,s)=(1,1)$.

Next, set
\begin{equation*}
\gamma_s^2
=
s^2-1+2b^2-3b^4.
\end{equation*}
By the Cauchy-Schwarz inequality, we have
\begin{equation}
\frac{|A_r^e|^2}{b^4}
\le
\widetilde C^e(b)
\sum_{s\in2\mathbb Z^+}
\frac{s^2\gamma_s^2}{b^4}
|A_{rs}|^2
\le
\widetilde C^e(b)
\sum_{s\in2\mathbb Z^+}
\widetilde d_{rs}|A_{rs}|^2,
\label{eq:tilde-even-A}
\end{equation}
\begin{equation}
\frac{|A_r^o|^2}{b^4}
\le
\widetilde C^o(b)
\sum_{s\in2\mathbb N+1}
\frac{s^2\gamma_s^2}{b^4}
|A_{rs}|^2
\le
\widetilde C^o(b)
\sum_{s\in2\mathbb N+1}
\widetilde d_{rs}|A_{rs}|^2,
\label{eq:tilde-odd-A}
\end{equation}
where $
\widetilde C^e(b)
=
\sum_{s\in2\mathbb Z^+}\frac{1}{\gamma_s^2}$ and $
\widetilde C^o(b)
=
\sum_{s\in2\mathbb N+1}\frac{1}{\gamma_s^2}$. Similarly, it follows that
\begin{equation}
\frac{|B_r^e|^2}{b^4}
\le
\widetilde C^e(b)
\sum_{s\in2\mathbb Z^+}
\widetilde d_{rs}|B_{rs}|^2,
\label{eq:tilde-even-B}
\end{equation}
\begin{equation}
\frac{|B_r^o|^2}{b^4}
\le
\widetilde C^o(b)
\sum_{s\in2\mathbb N+1}
\widetilde d_{rs}|B_{rs}|^2.
\label{eq:tilde-odd-B}
\end{equation}

Using the Mittag-Leffler expansions of $\tan x$ and $\cot x$, and observing that $0<b<\sqrt{\frac{2}{3}}$, $\sqrt{\frac{2}{3}}\le\BB:=\sqrt{1-2b^2+3b^4}<1$, we obtain
\begin{equation*}
\begin{split}
\widetilde C^e(b)
&=
\sum_{s\in2\mathbb Z^+}
\frac{1}{
s^2-\BB^2
}=
\frac12
\sum_{s\in2\mathbb Z}
\frac{1}{
s^2-\BB^2
}
+
\frac{1}{
2\BB^2
}\\
&=
-\frac18
\sum_{n\in\mathbb Z}
\frac{1}{
\frac{\BB^2}{4}-n^2
}
+
\frac{1}{
2\BB^2
}\\
&=
-\frac{\pi}{
4\BB
}
\cot\left(
\frac{
\pi\BB
}{2}
\right)
+
\frac{1}{
2\BB^2
}.
\end{split}
\end{equation*}
Moreover,
\begin{equation*}
\widetilde C^o(b)
=
\sum_{s\in2\mathbb N+1}
\frac{1}{
s^2-\BB^2
}
=
\frac{\pi}{
4\BB
}
\tan\left(
\frac{
\pi\BB
}{2}
\right).
\end{equation*}
Therefore, we get

\begin{equation}
\begin{split}
\widetilde C^o(b)-\widetilde C^e(b)
&=
\frac{\pi}{
4\BB
}
\left[
\tan\left(
\frac{
\pi\BB
}{2}
\right)
+
\cot\left(
\frac{
\pi\BB
}{2}
\right)
\right]
-
\frac{1}{
2\BB^2
}=
\frac{1}{
2\BB^2
}
\left(
\frac{
\pi\BB
}{
\sin\left(
\pi\BB
\right)
}
-1
\right)
>0.
\end{split}
\label{eq:Cto-Co-Ce}
\end{equation}

Combining \eqref{eq:sum-Ar},
\eqref{eq:tilde-even-A}, \eqref{eq:tilde-odd-A},
\eqref{eq:tilde-even-B}, \eqref{eq:tilde-odd-B}, and
\eqref{eq:Cto-Co-Ce}, we obtain
\begin{equation*}
\begin{split}
\widetilde D
&\ge
\sum_{(r,s)\in S_2}
\widetilde d_{rs}a_{rs}\ge
\frac{1}{\widetilde C^e(b)b^4}
\sum_{r\in2\mathbb N+1}
\left(
|A_r^e|^2+|B_r^e|^2
\right)
+
\frac{1}{\widetilde C^o(b)b^4}
\sum_{r\in2\mathbb N+1}
\left(
|A_r^o|^2+|B_r^o|^2
\right)\\
&\ge
\frac{1}{\widetilde C^o(b)b^4}
\sum_{r\in2\mathbb N+1}
\left(
|A_r^e|^2
+
|A_r^o|^2
+
|B_r^e|^2
+
|B_r^o|^2
\right)\\
&=
\frac{1}{\widetilde C^o(b)}
\frac{1}{8\pi^2b}.
\end{split}
\end{equation*}
Substituting the expression for $\widetilde C^o(b)$ into the preceding Willmore energy estimate gives
\begin{equation*}
\mathcal W(\psi_K)
\ge
\frac{\pi^2}{2}
\left(
4b+\frac{1}{b}
\right)
+
\frac{\pi\BB}{b}
\cot\left(
\frac{\pi\BB}{2}
\right).
\end{equation*}
The same argument as in the case $b>\sqrt{\frac{2}{3}}$ shows that equality cannot hold.
\end{proof}
\subsection{Auxiliary Lemmas in Section \ref{section5}}
\begin{lemma}\label{lemma-b-big}
 $J_1(b)$ is strictly increasing on
$\left(\sqrt{\frac{2}{3}},+\infty\right)$.
\end{lemma}
\begin{proof}
Set $\DD:=\sqrt{\frac{3b^2-2}{2}}$. Differentiating $J_1(b)$ gives
\begin{equation*}
\begin{split}
J_1'(b)
={}&
\frac{\pi^2}{2}
-\frac{3\pi^2}{2b^2}
-\frac{2\pi\DD}{b^2}
\tanh\left(
\frac{\pi\DD}{2}
\right)
+
\frac{3\pi}{\DD}
\tanh\left(
\frac{\pi\DD}{2}
\right)+
\frac{3\pi^2}{2}
\operatorname{sech}^2\left(
\frac{\pi\DD}{2}
\right)
\\
={}&
\frac{\pi^2}{2}\left(1-\frac{3}{b^2}\right)
+
\frac{2\pi}{b^2\DD}
\tanh\left(
\frac{\pi\DD}{2}
\right)
+
\frac{3\pi^2}{2}
\operatorname{sech}^2\left(
\frac{\pi\DD}{2}
\right),
\end{split}
\end{equation*}

Multiplying both sides by the positive factor $\frac{12b^2}{\pi^2}
\cosh^2\left(
\frac{\pi\DD}{2}
\right)$ and expanding $\sinh$ and $\cosh$ into their power series, we obtain
\begin{equation*}
\begin{split}
&\frac{12b^2}{\pi^2}
\cosh^2\left(
\frac{\pi\DD}{2}
\right)J_1'(b)
\\&=
6(b^2-3)
\cosh^2\left(
\frac{\pi\DD}{2}
\right)+
\frac{24}{\pi\DD}
\cosh\left(
\frac{\pi\DD}{2}
\right)
\sinh\left(
\frac{\pi\DD}{2}
\right)
+18b^2\\
&=
3(b^2-3)
+
3(b^2-3)
\cosh\left(\pi\DD
\right)+
\frac{12}{\pi\DD}
\sinh\left(\pi\DD
\right)
+18b^2
\\
&=
21b^2-9
+
(3b^2-9)
\left(
\sum_{n=0}^{\infty}
\frac{\pi^{2n}}{(2n)!}
\DD^{2n}
\right)
+
\frac{12}{\pi\DD}
\left(
\sum_{n=0}^{\infty}
\frac{\pi^{2n+1}}{(2n+1)!}\DD^{2n+1}
\right)
\\&=
5
+
14\DD^2
+
\left(
-7+2\DD^2
\right)
\sum_{n=0}^{\infty}
\frac{\pi^{2n}}{(2n)!}
\DD^{2n}
+
12
\sum_{n=0}^{\infty}
\frac{\pi^{2n}}{(2n+1)!}
\DD^{2n}
\\
&=
5
+
14\DD^2
-7
+
2\DD^2
+12
-
\frac{7\pi^2}{2!}\DD^2
+
\frac{12\pi^2}{3!}\DD^2+
\sum_{n=2}^{\infty}
\left(
\frac{2\pi^{2n-2}}{(2n-2)!}
-
\frac{7\pi^{2n}}{(2n)!}
+
\frac{12\pi^{2n}}{(2n+1)!}
\right)\DD^{2n}
\\
&=
10
+
\left(
16-\frac{3}{2}\pi^2
\right)\DD^2
+
\sum_{n=2}^{\infty}
\left[
\frac{\pi^{2n-2}}{(2n)!}
\left(
4n(2n-1)
-
\pi^2\left(
7-\frac{12}{2n+1}
\right)
\right)
\right]
\DD^{2n} .
\end{split}
\end{equation*}
For $n\geq4$, we get
$
4n(2n-1)
-
\pi^2\left(
7-\frac{12}{2n+1}
\right)
\geq
112-7\pi^2
>0.
$
Hence the terms in the above equation corresponding to $n\geq4$ are positive. Discarding the terms with $n\geq5$, and noting that
$
60-\frac{37}{7}\pi^2>0$ and $
112-\frac{17}{3}\pi^2>0$, 
the AM-GM inequality gives
\begin{equation*}
\begin{split}
&\frac{12b^2}{\pi^2}
\cosh^2\left(
\frac{\pi\DD}{2}
\right)
J_1'(b)\\&>
10
+
\left(
16-\frac{3\pi^2}{2}
\right)\DD^2
-
\frac{\pi^2(23\pi^2-120)\DD^4}{120}
+
\frac{\pi^4\DD^6}{720}
\left(
60-\frac{37}{7}\pi^2
\right)+
\frac{\pi^6\DD^8}{40320}
\left(
112-\frac{17}{3}\pi^2
\right)
\\
&\geq
\left[
2\sqrt{
\frac{10\pi^6}{40320}
\left(
112-\frac{17}{3}\pi^2
\right)
}
\right.
\left.
+
2\sqrt{
\frac{\pi^4}{720}
\left(
16-\frac{3\pi^2}{2}
\right)
\left(
60-\frac{37}{7}\pi^2
\right)
}
-
\frac{\pi^2(23\pi^2-120)}{120}
\right]\DD^4
\\
&=
\frac{\pi^2}{2520}
\left[
10\pi
\sqrt{
21(336-17\pi^2)
}
\right.
\left.
+
6
\sqrt{
70(32-3\pi^2)(420-37\pi^2)
}
-
21(23\pi^2-120)
\right]\DD^4
\\
&>0.
\end{split}
\end{equation*}
Hence $
J_1'(b)>0$,
when $b\in\left(\sqrt{\frac{2}{3}},+\infty\right)$. This proves the Lemma.
\end{proof}
We need the following lemma. The proof is straightforward and we omit it.
\begin{lemma}\label{lemma-P(x)}
Let $
P(x)
:=
\frac12+\frac{x}{\pi}\cot\frac{\pi x}{2}
$. Then for all $x\in\left[\sqrt{\frac23},1\right)$, we have
\begin{equation}
\frac12<P(x)<1-\frac{x}{2},\qquad
\label{eq:P-bound}
-\frac12<P'(x)<0, 
\qquad
-1<P''(x)<0,
\end{equation}
\end{lemma}

\begin{lemma}\label{lemma-b-small}
The function $J_2(b)$ has a unique critical point $b_*\in\left(\frac{1}{2},\frac{1}{\sqrt{3}}\right)$.
Moreover,
\[
J_2'(b)<0
\quad\text{for }0<b<b_*,
\qquad
J_2'(b)>0
\quad\text{for }b_*<b<\sqrt{\frac{2}{3}}.
\]
Consequently,
$
\min_{0<b<\sqrt{\frac{2}{3}}}J_2(b)
=
J_2(b_*).
$
\end{lemma}

\begin{proof}
We have now $
J_2(b)
=
\pi^2
\left(
2b+\frac{P(\BB(b))}{b}
\right).
$
Since
$
\BB'(b)
=
\frac{2b(3b^2-1)}{\BB}$, we have
\begin{equation}
J_2'(b)
=
\pi^2
\left(
2-\frac{P(\BB)}{b^2}
+
\frac{2(3b^2-1)}{\BB}P'(\BB)
\right).
\label{eq:J2-first}
\end{equation}

Next, we determine the sign of $J_2'(b)$ in two cases.

For $0<b\leq\frac{1}{\sqrt{3}}$, differentiating
\eqref{eq:J2-first} and applying Lemma \ref{lemma-P(x)} gives
\begin{equation}
\begin{split}
J_2''(b)
={}&
\frac{\pi^2}{b^3}
\left(
2P(\BB)
+
\frac{2b^2(1-b^2)(9b^4+1)}{\BB^3}
P'(\BB)
\right.
\left.
+
\frac{4b^4(1-3b^2)^2}{\BB^2}
P''(\BB)
\right)
\\
>{}&
\frac{\pi^2}{b^3}
\left(
1
-
\frac{b^2(1-b^2)(9b^4+1)}{\BB^3}
-
\frac{4b^4(1-3b^2)^2}{\BB^2}
\right).
\end{split}
\label{eq:J2-second}
\end{equation}

Let $\mathfrak{t}=b^2$, and set $
Q_1(\mathfrak{t})
=
\mathfrak{t}(1-\mathfrak{t})(9{\mathfrak{t}}^2+1)$, $
Q_2(\mathfrak{t})
=
\mathfrak{t}(1-3\mathfrak{t})
$. It is immediate that $Q_1(\mathfrak{t})$ is increasing for $\mathfrak{t}\in\left(0,\frac13\right]$, so we get
\begin{equation}
Q_1(\mathfrak{t})
\leq
Q_1\left(\frac13\right)
=
\frac49.
\label{eq:J2-Q1}
\end{equation}
On the other hand, for $\mathfrak{t}\in\left(0,\frac13\right)$, we have
\begin{equation}
0\leq Q_2(\mathfrak{t})\leq\frac{1}{12}.
\label{eq:J2-Q2}
\end{equation}
Substituting \eqref{eq:J2-Q1} and \eqref{eq:J2-Q2} into
\eqref{eq:J2-second}, and using
$\BB\in\left[\sqrt{\frac23},1\right)$, we obtain
\begin{equation}
\begin{split}
J_2''(b)
\geq{}
\frac{\pi^2}{b^3}
\left(
1
-
\frac{\frac49}{\left(\sqrt{\frac23}\right)^3}
-
\frac{4\left(\frac1{12}\right)^2}
{\left(\sqrt{\frac23}\right)^2}
\right)
=
\frac{\pi^2}{b^3}
\left(
1-\sqrt{\frac23}-\frac1{24}
\right)
>0.
\end{split}
\label{eq:J2-convex}
\end{equation}

For every $b\in\left(0,\frac1{\sqrt3}\right)$. When $b=\frac{1}{2}$, we also have
\[
J_2'\left(\frac12\right)
=
\pi^2
\left(
-\frac{13}{\pi\sqrt{11}}
\cot\frac{\pi\sqrt{11}}{8}
+
\frac14
\csc^2\frac{\pi\sqrt{11}}{8}
\right)
<0.
\]
On the other hand, substituting $b=\frac1{\sqrt3}$ into
\eqref{eq:J2-first} gives
\[
J_2'\left(\frac1{\sqrt3}\right)
=
\pi^2
\left(
\frac12
-
\frac{\sqrt6}{\pi}
\cot\frac{\pi\sqrt6}{6}
\right)
>0.
\]
By \eqref{eq:J2-convex}, there exists a unique $b_*
\in
\left(
\frac12,\frac1{\sqrt3}
\right)
\subset
\left(
0,\frac1{\sqrt3}
\right)$
such that $J_2'(b_*)=0$. 

It remains to prove that $J_2'(b)>0$ for
$\frac1{\sqrt3}\leq b<\sqrt{\frac23}$.
By Lemma \ref{lemma-P(x)}, \eqref{eq:J2-first} gives
\begin{equation}
\begin{split}
J_2'(b)
>
\pi^2
\left(
2
-
\frac{1-\frac{\BB}{2}}{b^2}
-
\frac{3b^2-1}{\BB}
\right)
=
\frac{\pi^2}{b^2\BB}
\left(
(2b^2-1)\BB
+
\frac{1-3b^4}{2}
\right).
\end{split}
\label{eq:J2-large-b}
\end{equation}
Thus it suffices to show that
$
E(b)
:=
(2b^2-1)\BB
+
\frac{1-3b^4}{2}
>0.$
Differentiating $E(b)$ gives
\begin{equation*}
\begin{split}
E'(b)
={}&
4b\BB
+
\frac{2b(2b^2-1)(3b^2-1)}
{\BB}
-
6b^3
=\frac{6b}{\BB}
\left(
4b^4-3b^2+1
-
b^2\BB
\right).
\end{split} 
\end{equation*}
It is direct to show that $4b^4-3b^2+1
>
b^2\BB$ and
therefore $E'(b)>0$ for
$b\in\left[\frac1{\sqrt3},\sqrt{\frac23}\right)$.

Consequently, we obtain
\[
E(b)
\geq
E\left(\frac1{\sqrt3}\right)
=
-\frac13\sqrt{\frac23}
+
\frac13
=
\frac13\left(1-\sqrt{\frac23}\right)
>0,
\qquad
b\in
\left[
\frac1{\sqrt3},
\sqrt{\frac23}
\right).
\]
Substituting this into \eqref{eq:J2-large-b} yields $
J_2'(b)>0,$ 
$b\in
\left[
\frac1{\sqrt3},
\sqrt{\frac23}
\right)$. 

Combining the two cases, there exists a unique
$b_*\in\left(\frac12,\frac1{\sqrt3}\right)$ such that
$J_2'(b_*)=0$, that is,
\[
\min_{0<b<\sqrt{\frac23}}J_2(b)
=
J_2(b_*).
\]
\end{proof}

\begin{lemma}\label{lem:J2-lower-bound}
We have that $J_2(b_*)>\frac{\pi^2}{2}\sqrt{5+4\sqrt{11}}.$
\end{lemma}

\begin{proof}
Lemma~\ref{lemma-b-small} shows that $\frac12<b_*<\frac1{\sqrt3}$. When $ \frac12<b<\frac1{\sqrt3}$, we have
$\BB'(b)=\frac{2b(3b^2-1)}{\BB(b)}<0.$
Set $\BB_*=\BB(b_*)$, we obtain
$
\sqrt{\frac23}<\BB_*<\frac{\sqrt{11}}4$. Since $
\cot\frac{\pi\BB}{2} = \tan\frac{\pi(1-\BB)}2 >
\frac{\pi(1-\BB)}2$, it follows that
\[
P(\BB)=\frac{1}{2}+\frac{\BB}{\pi}\cot\frac{\pi\BB}{2} 
>
\frac12+\frac{\BB(1-\BB)}2.
\]
Since $\BB_*>\frac{1}{2}$ and $\BB(1-\BB)$ is decreasing on
$(\frac{1}{2},1)$, we get 
$
P(\BB_*)
>
\frac12+
\frac12\frac{\sqrt{11}}4
\left(1-\frac{\sqrt{11}}4\right)
=
\frac{5+4\sqrt{11}}{32}.
$ Hence
${J_2(b_*)}
>{\pi^2}\left(
2b_*+
\frac{5+4\sqrt{11}}{32b_*}\right)\geq \frac{\pi^2}{2}\sqrt{5+4\sqrt{11}}.
$
\end{proof}
\normalfont
\vspace{2mm}
\textbf{Declaration on the use of AI.}
The AI tool ChatGPT 5.6 Sol was used in the proofs of Proposition \ref{lem:u-upper} in Section \ref{section4}, Lemma \ref{lem-piecewise-lower-bound} in Section \ref{section5}. The authors have independently verified all AI-assisted content, and take full responsibility for the contents of the paper. \\

$\mathbf{Funding~Declarations}$: PW is  supported by the Natural Science Foundation of Fujian Province of China (Grant No. 2026J002028) and  NSFC (Grant No. 12371052).\hfill \\

$\mathbf{Data~Availability~Declarations}$: The data that support the findings of this study are available from the corresponding author upon reasonable request.\hfill \\

$\mathbf{Ethics~And~Consent~To~Participate~Declarations}$: Not applicable - This research does not involve any human or animal participants.

\def\refname

{\normalsize\HT Reference}
\begin{thebibliography}{99}
\small\setlength{\parskip}{0pt}


\bibitem{B-H-K}
B\'{e}rard P. H., Helffer B., Kiwan R.
\emph{Courant-sharp eigenvalues of compact flat surfaces: Klein bottles and cylinders},
Proc. Amer. Math. Soc. {150} (2022), no. 1, 439--453.

\bibitem{B-H-M}
Breuning P., Hirsch J., M\"{a}der-Baumdicker E.
\emph{Existence of minimizing Willmore Klein bottles in Euclidean four-space},
Geom. Topol. {21} (2017), no. 4, 2485--2526.

\bibitem{B-1988}
Bryant R. L.
\emph{Surfaces in conformal geometry},
Proc. Sympos. Pure Math. {48} (1988), 227--240.

\bibitem{B-2015}
Bryant R. L.
\emph{On the conformal volume of 2-tori},
arXiv preprint arXiv:1507.01485 (2015).

\bibitem{BCGG}
Bujalance E., Cirre F. J., Gamboa J. M., Gromadzki G.
\emph{Symmetries of compact Riemann surfaces},
Lecture Notes in Mathematics, vol. 2007,
Springer, Berlin, Heidelberg, 2010.


\bibitem{Chen1981}
Chen B. Y.
\emph{On the total curvature of immersed manifolds, V: C-surfaces in Euclidean $m$-space},
Bull. Inst. Math. Acad. Sinica {9} (1981), no. 4, 509--516.

\bibitem{Chen2015}
Chen B. Y.
\emph{Total Mean Curvature and Submanifolds of Finite Type},
2nd ed., World Scientific, Singapore, 2015.


\bibitem{H-M}
Hirsch J., M\"{a}der-Baumdicker E.
\emph{A note on Willmore minimizing Klein bottles in Euclidean space},
Adv. Math. {319} (2017), 67--75.

\bibitem{J-N-P}
Jakobson D., Nadirashvili N., Polterovich I.
\emph{Extremal metric for the first eigenvalue on a Klein bottle},
Canad. J. Math. {58} (2006), no. 2, 381--400.

\bibitem{K-1989}
Kusner R. B.
\emph{Comparison surfaces for the Willmore problem},
Pacific J. Math. {138} (1989), no. 2, 317--345.

\bibitem{K-1996}
Kusner R. B.
\emph{Estimates for the biharmonic energy on unbounded planar domains,
and the existence of surfaces of every genus that minimize the
squared-mean-curvature integral},
in: Chow B., Gulliver R., Levy S., Sullivan J. M. (eds.), \emph{Elliptic and Parabolic Methods in Geometry},
A K Peters, Wellesley, MA, 1996, pp.~67--72.


\bibitem{Lawson-1970}
Lawson H. B.
\emph{Complete minimal surfaces in $S^3$},
Ann. of Math. {92} (1970), no. 3, 335--374.

\bibitem{L-Y}
Li P., Yau S.-T.
\emph{A new conformal invariant and its applications to the Willmore conjecture and the first eigenvalue of compact surfaces},
Invent. Math. {69} (1982), no. 2, 269--291.

\bibitem{L-X-W}
L\"{u} Y., Wang P., Xie Z.
\emph{Minimal isometric immersions of flat $n$-tori into spheres},
arXiv preprint arXiv:2504.13064 (2025).

\bibitem{M-N2}
Marques F. C., Neves A.
\emph{Min-max theory and the Willmore conjecture},
Ann. of Math. {179} (2014), no. 2, 683--782.

\bibitem{M-N1}
Marques F. C., Neves A.
\emph{The Willmore conjecture},
Jahresber. Dtsch. Math.-Ver. {116} (2014), no. 4, 201--222.


\bibitem{M-R}
Montiel S., Ros A.
\emph{Minimal immersions of surfaces by the first eigenfunctions and conformal area},
Invent. Math. {83} (1986), no. 1, 153--166.

\bibitem{T-1941}
Tompkins C. B.
\emph{A flat Klein bottle isometrically embedded in Euclidean $4$-space},
Bull. Amer. Math. Soc. {47} (1941), no. 6, 508.

\bibitem{WangWang2022}
Wang C., Wang P.
\emph{A survey on Willmore surfaces and Willmore conjecture in $n$-dimensional sphere},
J. Fujian Normal Univ. (Nat. Sci. Ed.) {38} (2022), 1--10.


\bibitem{W-1965}
Willmore T. J.
\emph{Note on embedded surfaces},
An. \c{S}tiin\c{t}. Univ. ``Al. I. Cuza'' Ia\c{s}i Sec\c{t}. I a Mat.
{11B} (1965), 493--496.
\end{thebibliography}
 \end{document}